\documentclass{article}
\usepackage{PRIMEarxiv}

\usepackage[utf8]{inputenc} 
\usepackage[T1]{fontenc}    
\usepackage{hyperref}       
\usepackage{url}            
\usepackage{booktabs}       
\usepackage{amsfonts,amsmath,amssymb,amsthm}       
\usepackage{nicefrac}       
\usepackage{microtype}      
\usepackage{lipsum}
\usepackage{fancyhdr}       
\usepackage{graphicx}       
\graphicspath{{media/}}     

\newtheorem{theorem}{Theorem}

\newtheorem{lemma}[theorem]{Lemma}

\newcommand{\bfm}[1]{\boldsymbol{#1}}

\title{A Neural Network Framework for Discovering Closed-Form Solutions to Quadratic Programs with Linear Constraints
\thanks{A portion of this paper appeared previously in the author’s PhD thesis: Fuat Can Beylunio\u{g}lu, ``Towards Explainable Neural Networks for Mathematical Programming" PhD Dissertation, University of Waterloo, 2025.} 
}

\author{
  Fuat Can Beylunio\u{g}lu, P. Robert Duimering, Mehrdad Pirnia \\
  Management Science and Engineering \\
  University of Waterloo \\
  Waterloo\\
  \texttt{\{fcbeylun,rduimering,mpirnia\}@uwaterloo.ca} \\
}

\usepackage{tcolorbox}
\newtcolorbox{remarkbox}{colback=black!10!white, colframe=white!50!black, title=Remark}

\usepackage{algorithm}
\usepackage{algpseudocode}
\usepackage{tikz}
\usepackage{booktabs}
\usepackage{multirow}

\usepackage{natbib}
 \bibpunct[, ]{(}{)}{,}{a}{}{,}%
\begin{document}
\maketitle

\begin{abstract}
Deep neural networks (DNNs) have been used to model complex optimization problems in many applications, yet have difficulty guaranteeing solution optimality and feasibility, despite training on large datasets. Training a NN as a surrogate optimization solver amounts to estimating a global solution function that maps varying problem input parameters to the corresponding optimal solutions. Work in multiparametric programming (mp) has shown that solutions to quadratic programs (QP) are piece-wise linear functions of the parameters, and researchers have suggested leveraging this property to model mp-QP using NN with ReLU activation functions, which also exhibit piecewise linear behaviour. This paper proposes a NN modeling approach and learning algorithm that discovers the exact closed-form solution to QP with linear constraints, by analytically deriving NN model parameters directly from the problem coefficients without training. Whereas generic DNN cannot guarantee accuracy outside the training distribution, the closed-form NN model produces exact solutions for every discovered critical region of the solution function. To evaluate the closed-form NN model, it was applied to DC optimal power flow problems in electricity management. In terms of Karush-Kuhn-Tucker (KKT) optimality and feasibility of solutions, it outperformed a classically trained DNN and was competitive with, or outperformed, a commercial analytic solver (Gurobi) at far less computational cost. For a long-range energy planning problem, it was able to produce optimal and feasible solutions for millions of input parameters within seconds.
\end{abstract}


\maketitle


\section{Introduction}\label{sec:Intro}

Neural Networks (NN) have been used as surrogate optimization models to predict solutions to optimization problems in energy management, chemistry, control theory, and other fields \citep{nellikkath2022physics,karg2020efficient,katz2020integrating}. NNs are attractive as optimization solvers, due to their speed and computational efficiency advantages compared to analytic solvers. Conceptually, the goal of such a model is to estimate a solution function that maps varying input parameters, such as different values of the right hand side (RHS) vector or cost coefficients for a given optimization problem, to the corresponding optimal solutions. Typical applications use a generic deep NN (DNN) trained on datasets of input-output pairs representing particular parameter values and optimal solutions, respectively, to approximate solutions. Such black box models ignore the mathematical structure of the underlying solution function and require large, computationally costly training sets to achieve satisfactory performance.

Mathematical programming has primarily been a study of numeric optimization where the optimal solution is obtained iteratively, such as using Quasi Newton-Raphson approaches and the simplex method. This is mainly due to the fact that the closed-form solution function is often difficult to obtain in constraint optimization. The potential of NN models to predict accurate solutions to optimization problems without computationally costly solver iterations, has significant implications for operations research theory and practice. A NN model trained on pre-solved instances effectively acts as a proxy of a closed-form solution function, i.e., taking input parameters and producing near-optimal solutions. Therefore, it is important to understand (i) the mathematical properties of the solution function and (ii) the capability of NN models to represent this function.

The properties of the solution function have been studied in the multiparametric programming and control literature\footnote{This literature uses the term optimizer to refer to the function mapping parameters to optimal primal solutions. We use the term solution function to refer to the mapping between parameters and both primal and dual solutions.}. For multiparametric linear programs (mp-LP) and quadratic programs (mp-QP), it has been shown that a function, $g$, defined between unknown parameters added to the RHS of constraints, $\bfm \theta$, to primal solutions, $\textbf{x}^*$, i.e., $g: \bfm\theta\rightarrow\textbf{x}^*$, has a piecewise affine, or piecewise linear (PWL) form \citep{pistikopoulos2020multi}. A NN with rectified linear unit (ReLU) activation functions also has the form of a PWL function with trainable weights and biases \citep{montufar2014number}. This common PWL structure raises the theoretical possibility of finding optimal NN weights and biases such that the model provides an exact closed-form representation of the solution function for any mp-LP or mp-QP \citep{karg2020efficient}.


However, training a DNN model that guarantees solution optimality and feasibility is very difficult in practice. One challenge arises from the nonconvex nature of the training loss function, since different combinations of NN weights and biases may yield different local minima for the same dataset. Furthermore, model prediction accuracy depends strongly on the composition of the training set and declines when inputs fall outside the training distribution. To represent the true global solution function, therefore, it would be necessary to train with extremely large datasets, spanning all critical regions of the parameters, which quickly becomes unmanageable as dimensionality increases. For example, in an \( n \)-dimensional space, including just boundary points for every feasible region implies the need for \( 2^n \) samples. Yet, not even this is enough for accurate estimation, since data points inside each region are also required \citep{beylunioglu2025partially}.

To address these challenges, this paper proposes a NN modeling approach for mp-QP with linear constraints and a learning-via-discovery (LvD) algorithm that derives all model parameters directly from the problem coefficients without the need for training. The proposed approach calculates exact model parameters from the problem coefficients via linear algebraic calculations, rather than approximating  parameters by minimizing loss over a training dataset. Thus, model predictions do not depend on dataset sampling and generalize to every feasible input parameter, given that all critical regions in the feasible domain of the solution function are discovered by the LvD algorithm. In other words, our proposed NN model provides an exact representation of the closed-form solution function to mp-QP with linear constraints for all discovered critical regions.


The proposed closed-form NN model is designed to expand as new critical regions of the solution function are discovered, by fixing the first layer to the slopes of each  discovered linear segment and using the next layers to choose the correct slope. The LvD algorithm learns the NN model parameters by incrementally discovering new critical regions, corresponding to different sets of binding constraints, and gradually expands the model to include parameters computed based on the slope of each new region. The algorithm starts at a known feasible input parameter, $\bfm\theta^0$, and the corresponding binding constraints at the solution, $\mathcal{B}_0$. The NN is initialized only with the slope of the linear segment in this critical region, by fixing the first layer to $\textbf{W}^0 = [\nabla\bfm\mu^*_{\mathcal{B}_0}]$. This initial model already includes all the information needed to generate optimal solutions within the initial critical region, so it is used to predict solutions by incrementally adjusting input parameters until a parameter, $\bfm\theta^{j+1}$, is found such that the current state of the model violates KKT optimality conditions, indicating that a new critical region has been discovered. Then, the set of binding constraints in the new region, $\mathcal{B}_{j+1}$, is identified by either adding a new constraint to the current set or deleting one that is no longer binding, and the NN model is expanded by adding the corresponding slope $\nabla \bfm\mu^*_{\mathcal{B}_{j+1}}$ to $\textbf{W}^0$. This procedure is repeated until all critical regions are identified and included in the model.

Our approach differs from the literature and contributes as follows:
\begin{enumerate}
	\item We propose a NN modeling approach and learning-via-discovery algorithm that can produce exact optimal and feasible solutions to mp-QP with uncertain RHS values applied on the equality constraints. To the best of our knowledge, our model is the first to represent closed-form mp-QP solution functions as a NN, to the extent that the discovery step visits all critical regions. Although prior work has recognized functional similarities between ReLU NN models and MP solutions, existing studies use training data to approximate the solution function.
    \item The proposed LvD algorithm is fundamentally different from traditional NN training, with the goal of discovering critical regions rather than estimating model weights from data. The result is an entirely explainable white-box NN model, whose weights correspond directly to parameters of the solution function. Because the initialized model expands as critical regions are discovered, at any learning iteration, the violations made by its current state are used to infer the binding constraints of new regions, without the need for analytic solvers.  
    \item Our model learns a single global solution function that directly represents the true solution mapping for all discovered critical regions. Hence, when predicting new solutions it is agnostic to the region in which input parameters reside. By contrast, traditional multiparametric programming solutions are an ensemble of multiple solution functions, so an extra step of first determining the critical region is needed to apply the mapping that corresponds to the correct active constraint set.
    \item Our model guarantees the optimality and feasibility of solutions for discovered critical regions, with accuracy limited only by the floating point precision of the NN package employed. For example, with 64-bit representation, solutions exceed the accuracy of commercial solvers such as Gurobi.
    

\end{enumerate}

The paper is organized as follows. Section \ref{sec:lit} discusses related literature and Section \ref{sec:mpqp} describes the mp-QP problem, derives abstract solution strategies, and provides proofs on properties of mp-QP solutions needed to support the proposed approach. Section \ref{sec:methodology} describes the model architecture, introduces the LvD algorithm, and explains it using a 2D example. Section \ref{sec:results} compares model precision and speed against alternative approaches using IEEE test sets from the domain of electricity distribution networks. The last section discusses limitations, conclusions, and future work.

\section{Related Work}\label{sec:lit}


\subsection{ML for Optimization}

Machine learning research for optimization falls into two broad categories: using ML to improve traditional optimization methods, and training ML models for end-to-end optimization learning. The first group uses ML for heuristics to improve solution time and efficiency of existing optimization algorithms, such as using NNs to learn branching rules for MILP problems \citep{gupta2020hybrid}, using reinforcement learning to predict optimal cuts for MILP problems  \citep{tang2020reinforcement} or optimal combinations for branching rules \citep{balcan2018learning} for a general class of problems. Other work employs statistical learning approaches as heuristics for combinatorial optimization, such as learning active constraints to reduce problem dimensions by removing inactive constraints \citep{baker2018joint}, or to provide more accurate predictions and quicker solutions \citep{misra2021learning,ng2018statistical, deka2019learning,chen2022learning}. ML also offers computational advantages for column generation as a heuristic model to predict the reduced cost of columns without directly solving the pricing subproblems \citep{shen2022enhancing} or to find paths leading to fewer iterations in the master problem \citep{chi2022deep}.

The use of ML for end-to-end optimization learning leverages DNNs to approximate the complex, nonlinear mapping between inputs and optimal decision variables, with applications in various fields \citep{karg2020efficient,katz2020integrating}. Surrogate NN optimization models promise significant speed advantages over analytical solvers, so they are particularly useful in electricity power management, where complex problems must be solved repeatedly with varying input parameters \citep{nellikkath2022physics,fioretto2020predicting,lotfi2022constraint}. Despite their potential benefits, these models still lack strong guarantees of optimality and feasibility, and require large datasets and substantial computational power for training.



To improve feasibility, some researchers proposed two-stage models that first predict optimal solutions with a neural network, followed by post-processing to ensure feasibility \citep{zamzam2020learning,chen2023end}. Another study used a sigmoid activation function in the NN output layer to impose generator dispatch limits  \citep{lotfi2022constraint}. \citet{fioretto2020predicting} formulated optimal solutions through sequentially connected sub-networks designed to reflect key characteristics of the problem. Other work proposed using physics-informed neural networks (PINNs) trained to satisfy the Karush-Kuhn-Tucker (KKT) optimality conditions, resulting in smaller dataset requirements, higher prediction accuracy with fewer violations, and improved generalizability \citep{nellikkath2022physics}. 

Despite these advances, most approaches still rely on treating the neural network model as a black-box that approximates the functional mapping from inputs to optimal solutions, and thus do not leverage the underlying mathematical structure of the optimization problem. For this reason, these models struggle to generalize beyond their training data and cannot guarantee feasibility or optimality. In the next section, we outline the multiparametric programming literature that investigates the functional relationship between inputs and optimal solutions.

\subsection{Multiparametric Programming}
Multiparametric programming is a subfield of mathematical optimization that aims to characterize how optimal solutions and cost function vary with respect to a set of parameters applied to the problem coefficients. MP addresses the problem as finding the solution function that produces optimal decision variables as explicit functions of these parameters. For LP, MP solutions yield PWL mappings over polyhedral regions in the parameter space. The geometric foundation for mp-LP accounts for critical regions as polytopes, with the optimizer function updates as the parameter vector moves among adjacent regions \citep{borrelli2003geometric}. This geometric methodology extends to mp-QP, where the optimizer remains PWL and the value function becomes piecewise quadratic in the parameters \citep{bemporad2000explicit, pistikopoulos2020multi}. The theory of MP uses the parametric simplex algorithm for mp-LP and the Basic Sensitivity Theorem for mp-QP \citep{fiacco1976sensitivity}, which describes how solutions change when parameters change. MP can be generalized to the mixed-integer setting as unions of multiple mp-LP or mp-QP instances, distinguished by discrete variable assignments (see \cite{oberdieck2016multi} for a review).

The feasible parameter space of an mp-QP can be systematically partitioned into critical regions by enumerating all possible active sets and removing infeasible configurations. However, because the number of active sets grows exponentially with the number of problem dimensions, exhaustive enumeration quickly becomes impractical for larger problems. The process can be made more efficient by identifying infeasible active sets early—since any superset of an infeasible configuration is itself infeasible—leading to the use of branch-and-bound algorithms to help discover admissible critical regions more effectively \citep{gupta2011novel,feller2012combinatorial,feller2013explicit}.

A complementary approach uses a bottom-up strategy for critical region discovery. Here, the process starts by solving the problem for an initial parameter value $\bfm\theta_0$ to identify the active set and deriving the associated parametric solution. The exploration of the parameter space then follows the adjacency graph of the critical regions. Earlier techniques, such as constraint reversal—which involves changing the membership of each constraint in the active set as in \citet{bemporad2000explicit}—face scalability issues due to the rapid growth of artificial cuts and the lack of completeness guarantees. To improve scalability, methods like the variable step size algorithm have been developed to infer neighboring regions by navigating along region facets, finding the center, and moving to adjacent polytopes \citep{baotic2003new}. Nevertheless, this approach has limitations when the facet-to-facet adjacency property does not hold in the feasible parameter set.

The PWL structure of mp-LP and mp-QP solutions exhibits functional similarities with DNNs using ReLU activation functions. This parallel has inspired researchers to use DNNs as surrogate models to represent the solution mappings in such problems. NN models offer distinct advantages for multiparametric optimization, as traditional approaches require explicitly determining the critical region corresponding to a given input parameter $\theta$ in order to apply the correct solution mapping associated with its active set. In large-scale applications, the number of regions can grow exponentially, so this ``point location problem" becomes a bottleneck, causing serious demands on both memory and the computation time needed to generate and store the full solution. Because of these limitations, exact multiparametric approaches are typically restricted to problems of moderate size. By contrast, NN solvers are agnostic to the critical region in which input parameters reside, as the model can theoretically represent the global solution function over all discovered critical regions. NN models offer additional speed advantages thanks to the parallelized nature of the calculations.

\cite{karg2020efficient} trained DNNs to emulate model predictive control (MPC) laws and approximate MPC solutions for time-varying systems. \cite{huo2022integrating} introduced an approach that integrates MPC laws for MILP microgrid control problems using NN models. In both cases, the DNNs were used as black-box approximators rather than deriving explicit solution mappings, and the optimality and feasibility of their models' predictions were not reported. \cite{chen2024physics} proposed a PINN-based approach that trains a NN model to solve equality constrained mp-QP problems using a loss function derived from the problem coefficients to penalize equality constraint violations. The model reduced constraint violations to as small as 1E-07, but overall validation loss did not fall below 1E-03 for a test case with 15 decision variables and two equality and no inequality constraints.  

Other research has used NN modeling to improve analytical mp solvers. \cite{katz2020integrating} employed NN models to approximate nonlinear functions that are difficult to integrate with MILP solvers. They reformulated a DNN as a MILP, enabling the use of MILP solvers to address the original nonlinear problem. \cite{rahal2025decision} introduced a decision rule approach that applies the max (or ReLU) function directly to uncertain parameters to generate new features. These are combined with the original parameters to form a richer decision rule, simplifying the mp-LP model and reducing computational and memory requirements.

\cite{beylunioglu2025partially} investigated difficulties in training NN models to represent the QP solution function exactly and proposed a partially-supervised NN model and training procedure. They derived the first layer weights from the underlying problem algebraically, followed by training on a relatively small dataset to learn final weights and biases. The model outperformed traditional NN approaches, but had certain limitations. Training required a large number of epochs to converge for problems with a large number of variables and inequality constraints. Prior knowledge of critical regions was also needed to initiate the model, which required discovery by solving the problem repeatedly on a large number of input parameters that increased with system size.


To address limitations of prior work, this paper proposes a custom NN model that exactly represents the PWL solution function and a learning algorithm to derive model weights analytically without training, while discovering critical regions within the feasible domain without using solvers. The proposed model provides an exact closed-form representation of the mp-QP solution function over all discovered critical regions. We apply the approach to model systems with up to 63 primal and 71 dual variables for proof of concept. The approach avoids the point-location problem and, due to parallelized NN calculations, offers dramatic computational advantages over analytic solvers without sacrificing solution optimality and feasibility. 



\section{Multiparametric QP with Linear Constraints} \label{sec:mpqp}
Typically, mp-QP problem with linear constraints  involves an unknown term in the RHS of constraints or added to the cost function coefficients. The problem can be defined as
\begin{subequations}
	\begin{align}
		\textbf{Minimize }  &\  z(\bfm\theta) = \textbf{x}^T\textbf{Q}\textbf{x}+(\textbf{C}+\bfm\theta_c)^T\textbf{x}+\textbf{C}_0,\\
		\textbf{s.t. } &\textbf{A}_e\textbf{x}=\textbf{b}_e + \bfm\theta_e, \quad [\lambda]\\
		&\textbf{A}_\mathcal{C}\textbf{x}\leq \textbf{b}_\mathcal{C} + \bfm\theta_\mathcal{C}, \quad [\mu]
	\end{align}
	\label{eq:ineq}
\end{subequations} 
where $\textbf{x} \in \mathbb{R}^{n}, \textbf{C} \in \mathbb{R}^{n}, \textbf{C}_0 \in \mathbb{R}^{n}, \textbf{b}_e\in \mathbb{R}^{m_1}, \textbf{b}_\mathcal{C}\in \mathbb{R}^{m_2}$, $\textbf{Q} \in \mathbb{R}^{n\times n}$ is a positive semidefinite matrix, $\textbf{A}_e\in \mathbb{R}^{m_1\times n}$, $\textbf{A}_\mathcal{C}\in \mathbb{R}^{m_2\times n}$, and $\bfm\theta \in \Theta_f$, where $\Theta_f$ is the set of all uncertain feasible parameters. Therefore, (\ref{eq:ineq}) is parametrized by $\bfm\theta = [\bfm\theta_c,\bfm\theta_e,\bfm\theta_\mathcal{C}]$ where $\bfm\theta_c \in \mathbb{R}^n, \bfm\theta_e\in \mathbb{R}^{m_1}, \bfm\theta_\mathcal{C} \in\mathbb{R}^{m_2}$. Here, the subscript of $\textbf{A}_\mathcal{C}$ and $\textbf{b}_\mathcal{C}$ is the set of all inequality constraint indices, i.e., $\mathcal{C}=\{1,2,\dots,m_2\}$. The Lagrangian function of the above problem can be written as,
\begin{align}
	L(\textbf{x},\boldsymbol{\lambda},\boldsymbol{\mu}) &= \textbf{x}^T\textbf{Q}\textbf{x}+(\textbf{C}+\bfm\theta_c)^T\textbf{x}+\textbf{C}_0 + 
	\boldsymbol{\lambda}^T( \textbf{b}_e+\bfm\theta_e-\textbf{A}_e\textbf{x})+
	\boldsymbol{\mu}^T( \textbf{b}_\mathcal{C}+\bfm\theta_\mathcal{C}-\textbf{A}_\mathcal{C}\textbf{x}),
	\label{eq:lagr}
\end{align}
where $\boldsymbol{\lambda} \in \mathbb{R}^{m_1},\boldsymbol{\mu}\in \mathbb{R}^{m_2}$ are dual variables of equality and inequality constraints, respectively.

In cases where the inequality constraints are not binding at the solution, it is known that $\bfm \mu^*=0$. Then, the solution can be calculated by fixing $\bfm\mu=0$ in (\ref{eq:lagr}), setting the derivatives of the Lagrangian to zero, and solving the resulting linear system of equations. If one or more inequality constraints are binding, then $\bfm\mu^*>0$. In this case, under the assumption that the value of $\bfm\mu^*>0$ is known, the solution can be achieved using the same system of equations. Therefore, given the known value of $\bfm\mu^*$, we can define the following solver function (see Appendix \ref{sec:solution} for derivation).
\begin{align}
	g(\bfm\mu^*,\bfm\theta) = \textbf{J}^{-1}
\begin{bmatrix}
	-\textbf{C}-\bfm\theta_c+\textbf{A}_\mathcal{C}^T\bfm\mu^* \\ -\textbf{b}_e-\bfm\theta_e \qquad 
\end{bmatrix}.
	\label{eq:linearsolver}
\end{align}
Equation \ref{eq:linearsolver} is a function that produces the rest of the optimal primal and dual variables for the problem in (\ref{eq:ineq}a)-(\ref{eq:ineq}c), given that the optimal shadow prices of the inequality constraints, $\bfm\mu^*=\mu(\bfm\theta)$, are known. The next section details the piecewise affine nature of the solutions and shows that it is due to the PWL characteristics of $\mu(\bfm\theta)$.

\subsection{Global Properties of mp-QP Solutions}\label{sec:global}
As seen in (\ref{eq:linearsolver}), when optimal shadow prices $\bfm\mu^*$ are known, the remaining optimal values can be obtained through a linear mapping. However, it can be shown that changes in optimal solutions to the problem in (\ref{eq:ineq}),  $\textbf{x}^*,\bfm\lambda^*,\bfm\mu^*$ with respect to changes in $\bfm\theta$ have PWL characteristics. This implies that shifts in the slope of the solution function stem from the functional form of $\mu(\bfm\theta)$. \citet[p. 50]{pistikopoulos2020multi} proposed the following theorem, which  establishes that the optimal primal solutions have PWL characteristics and that the optimal objective function exhibits piecewise quadratic properties with respect to changes in the parameters, $\bfm\theta$.

\begin{theorem}\label{th:pwa}
	For the mp-QP problem \ref{eq:ineq}, $\Theta_f \subseteq \Theta$ is a convex set, the primal solution function \textbf{x}$(\bfm\theta): \Theta_f \rightarrow \mathbb{R}^n$ is continuous and piecewise linear. Also the optimal objective function $\textbf{z}(\bfm\theta):\Theta_f\rightarrow \mathbb{R}$ is continuous and piecewise quadratic.
\end{theorem}
\textit{Proof:} See Appendix \ref{sec:pwl_proof} for proof. A different proof is given by \citet[p. 50]{pistikopoulos2020multi}.

\begin{lemma}\label{th:PWLmu}
	Consider the mp-QP problem \ref{eq:ineq}. The shadow prices of inequality constraints, $\bfm\mu^*:\Theta_f \rightarrow \mathbb{R}^{m_2}$ are piecewise linear.
\end{lemma}
\textit{Proof:} The mapping $g(\cdot)$ in (\ref{eq:linearsolver}) is an affine mapping from $\bfm\mu^*=\mu(\bfm\theta)$, and as shown in Theorem \ref{th:pwa}, the output of $\textbf{x}(\bfm\theta)$ is a piecewise affine function of $\bfm\theta$. As $g(\cdot)$ can be defined as a composite function, i.e., $(g\circ \mu)(\bfm\theta)$, the only input of the function has to be piecewise affine.

Theorem \ref{th:pwa} and Lemma \ref{th:PWLmu} demonstrate that the piecewise linearity of the primal and dual solutions is mainly due to the PWL characteristics of $\bfm\mu^*=\mu(\bfm\theta)$, in contrast to the rest of the solutions that can be obtained with a linear mapping. As is evident from (\ref{eq:derivativeswithmu}), each $\bfm{\mathcal{B}}_i$ defines a feasible subset of $\bfm \theta$, called a \textit{critical region}, where solutions for all $\bfm\theta$ from the same critical region can be obtained by solving the same system of equations since the set of binding constraints does not change at the solution. When moving away from this interval, the system of equations (thus the slope of $\mu(\bfm\theta)$) updates with a different $\mathcal{B}$, resulting in piecewise linear form.

Figure \ref{fig:mu} uses a quadratic optimization problem with $x\in\mathbb{R}^8, \bfm\lambda\in\mathbb{R}^6,\bfm\mu \in \mathbb{R}^6$, to illustrate changes in the slope of $\mu(\bfm\theta)$ with respect to each critical region. The problem was solved repeatedly by varying the RHS of the equality constraints, increasing $\theta_0$ from 0 to the maximum feasible value with other parameters set to 0. It can be seen that there are 5 critical regions determined by which constraint is active at the solution (i.e., $\mu_i^*>0$), and the slope of $\mu(\bfm\theta)$ is different in each region. 

\begin{figure}
	\centering
	\includegraphics[width=.8\linewidth]{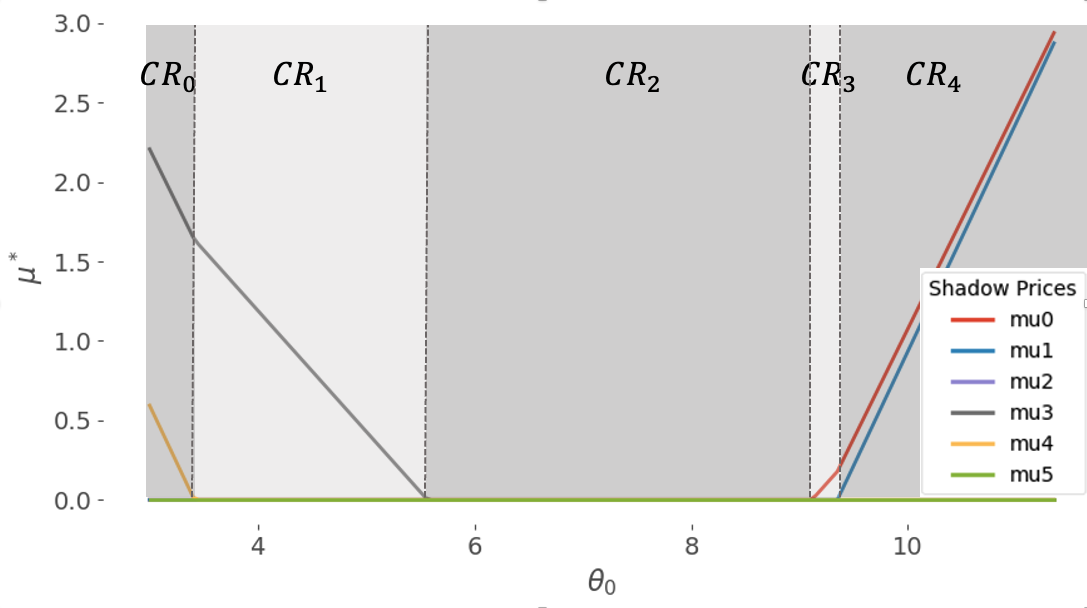}
	\caption{Sensitivity of $\boldsymbol{\mu}$ with respect to $ 
		\theta_0$}
	\label{fig:mu}
\end{figure}

		

Note that as $\bfm\theta$ passes from one critical region to another, either the shadow price of a constraint changes from 0 to positive or vice versa. In other words, at the intersection of two critical regions, the set of binding constraints either expands with a new constraint, or one of the binding constraints no longer binds. Moreover, at the intersection of the two critical regions, adjacent line segments produce the same value of the shadow price, an important property of mp-QP leveraged by the method proposed below.   \citet[p. 52]{pistikopoulos2020multi} proposed the following theorem, which shows that either one of the two cases occurs when transitioning between critical regions.

\begin{theorem}\label{th:transition}
	Consider an active set $\mathcal{B}_i  = \{k_1,k_2,\dots, k_{m}\}$ and its corresponding critical region $CR_i$ in minimal representation, i.e. with all redundant constraints removed. Also, let $CR_j$ be a full-dimensional adjacent critical region to $CR_i$ and assume independent constraint qualification holds for facet $F = CR_i\cap CR_j$, where $H$ is the separating hyperplane. Let $\mu(\bfm\theta)$ be the dual variables of the inequality constraints. Moreover, assume that there are no constraints, which are weakly active at the optimal solution $\textbf{x}^*=x(\bfm\theta)$ at $\forall \bfm\theta \in CR_i$. Then one of the following holds:
	\begin{itemize}
		\item Expanding: If H is given by $\textbf{A}_{m+1} x(\bfm\theta) = \textbf{b}_{k_{m+1}} + \bfm{\theta}_{m+1}$, where $\textbf{A}_{m+1},\textbf{b}_{m+1}$ are coefficients of the $k_{m+1}$th inequality constraint, then the optimal active set in $CR_j$ is $\{k_1,k_2,\dots,k_{m+1}\}$
		
		\item Diminishing: If $H$ is given by the shadow price of the $k_{m}$th constraint, $\mu_{k_m}(\bfm\theta)=0$, then the optimal active set in $CR_j$ is $\mathcal{B}_j = \{k_1,k_2,\dots,k_{m-1}\}$.
	\end{itemize}
\end{theorem}

\begin{proof}
	\ See \citet[p. 52]{pistikopoulos2020multi} for the proof.
\end{proof}

\subsection{Deriving Slopes of Shadow Price Function} \label{sec:muSlopes}
While there is no straightforward approach to formulate $\mu(\bfm\theta)$ analytically, the linear segment of this function corresponding with each critical region, $CR_i$, can be derived from the problem coefficients if the corresponding set of binding constraints, $\mathcal{B}_i$, is known. To find the slope of each linear segment of the function  with respect to the changes in $\bfm \theta$ for binding constraints, $\nabla \boldsymbol{\mu}^*_{\mathcal{B}_i}$, coefficients of binding inequality constraints, $\textbf{A}_{\mathcal{B}_i}$ is appended below the \textbf{J} matrix defined in (\ref{eq:linearinverse}) and then the resulting matrix is inverted.
\begin{align}
	\begin{bmatrix}
		\textbf{x}^*\\ \boldsymbol{\lambda}^* \\ \boldsymbol{\mu}^*_{\mathcal{B}_i}
	\end{bmatrix} 
	&=
	\begin{bmatrix}
		\nabla \textbf{x}^* \\ \nabla \boldsymbol{\lambda}^* \\ \nabla \boldsymbol{\mu}^*_{\mathcal{B}_i}
	\end{bmatrix}
	\begin{bmatrix}
		-\textbf{C}-\bfm\theta_c \\ -\textbf{b}_e -\bfm\theta_e \\ -\textbf{b}_{\mathcal{B}_i}-\bfm\theta_{\mathcal{B}_i}.
	\end{bmatrix}, 
	\label{eq:linearinverse_ineq}
\end{align}	
where $[\nabla \textbf{x}^*,  \nabla \boldsymbol{\lambda}^*,   \nabla \boldsymbol{\mu}^*_{\mathcal{B}_i}]^T = \textbf{J}_{\mathcal{B}_i}^{-1}$ and
\begin{align}
	\textbf{J}_{\mathcal{B}} = 
	\begin{bmatrix}
		2\textbf{Q} & -\textbf{A}_e^T & -\textbf{A}_{\mathcal{B}_i}^T\\
		-\textbf{A}_e & 0 & 0\\
		-\textbf{A}_{\mathcal{B}_i} & 0 & 0
	\end{bmatrix}.
\end{align}
The vector $\nabla \boldsymbol{\mu}^*_{\mathcal{B}_i}$, which consists of the slope parameters of the linear segment of the PWL solution function when the constraint set $\mathcal{B}_i$ binds, is obtained from the last row of the inverted Jacobian matrix, $\mathbf{J}^{-1}_{\mathcal{B}_i}$. That is, this vector is the set of coefficients for the mapping between $\mathbf{B} + \boldsymbol{\theta}$ and $\boldsymbol{\mu}^*_{\mathcal{B}_i}$, where
\begin{align}
	\textbf{B} + \bfm\theta= [\textbf{C} \ \ \textbf{b}_e \ \ \textbf{b}_\mathcal{C}]^T + [\bfm\theta_c \ \ \bfm\theta_e\ \ \bfm\theta_\mathcal{C}]^T.
\end{align}
Next, we propose a custom NN model and an analytical method to derive the model parameters from the optimization problem.


	\section{Methodology} \label{sec:methodology}
	
	\subsection{The Model}
	
	Our QP solver model implements the calculations described in Section \ref{sec:mpqp} in one NN flow, by calculating shadow prices of the inequality constraints and the rest of the solutions through separate parts of the model. As illustrated in Figure \ref{fig:model}, the model is a mapping from linear cost coefficients and RHS of all constraints (binding and nonbinding) to optimal primal and dual solutions, $f: \textbf{B}+\bfm\theta \rightarrow [\textbf{x}^*,\bfm\lambda^*,\bfm\mu^*]$. It consists of two subnetworks: 
    
     \textbf{Shadow Price Model, $f_\mu(\cdot)$}: The first subnetwork, $f_\mu(\cdot)$ is a three-layer DNN model with a custom architecture that  estimates $\mu(\bfm\theta)$, whose PWL properties are discussed in Section \ref{sec:global}. This model is detailed in Section \ref{sec:shadowPriceModel}.
    
    \textbf{Solution Model, $f_g(\bfm\mu^*,\bfm\theta)$}: The second subnetwork, $f_g(\bfm\mu^*,\bfm\theta)$, is a linear model in the form of $f_g(\textbf{z}) = \textbf{W}\textbf{z}+\textbf{b}$ that estimates $g(\bfm\mu^*,\bfm\theta)$ in (\ref{eq:linearsolver}). Note that the parameters of $f_g(\bfm\mu^*,\bfm\theta)$ are already derived in  (\ref{eq:linearsolver}), where $\textbf{W} = \textbf{J}^{-1}$ and $\textbf{b}=0$. 
       
    	While parameters of $f_g(\bfm\mu^*,\bfm\theta)$ are easy to derive, there is no straightforward method to derive the parameters of $f_\mu(\cdot)$. Our methodology is built upon on the observations in Section \ref{sec:global} that within each critical region, $CR_i$, the change in shadow prices with respect to the parameters, $\mu(\bfm\theta)$ for $\bfm\theta \in CR_i$, is a separate linear function whose parameters can be derived from the problem coefficients using (\ref{eq:linearinverse_ineq}). The next section describes how one can combine these separate functions to form a PWL NN model. Then, in Section \ref{sec:shadowPriceModel}, we propose a NN model that forms this PWL function by choosing the correct $\bfm\mu_i^*(\bfm \theta)$ for any $\bfm\theta$. 
	
	\begin{figure}[!h]
		\centering
		\includegraphics[scale=.25]{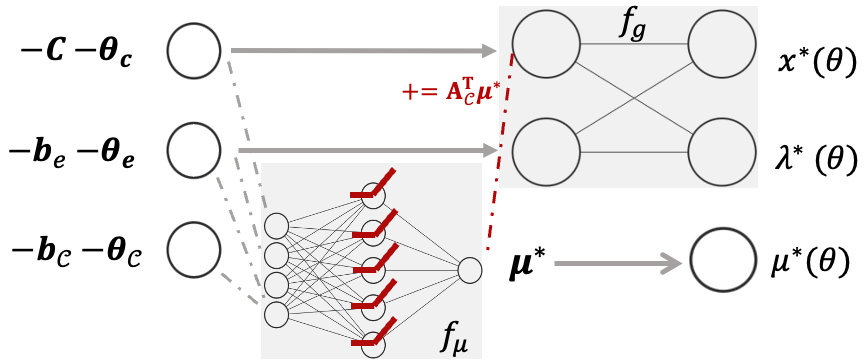}
		\caption{The QP Solver Model Architecture}
		\label{fig:model}
	\end{figure}
	
	\subsection{Modeling PWL Changes}\label{sec:PWLchanges}
	
	\begin{figure}[!h] 
		\centering
		\begin{minipage}[t]{.45\linewidth}
			\centering
			\includegraphics[width=1.2\linewidth]{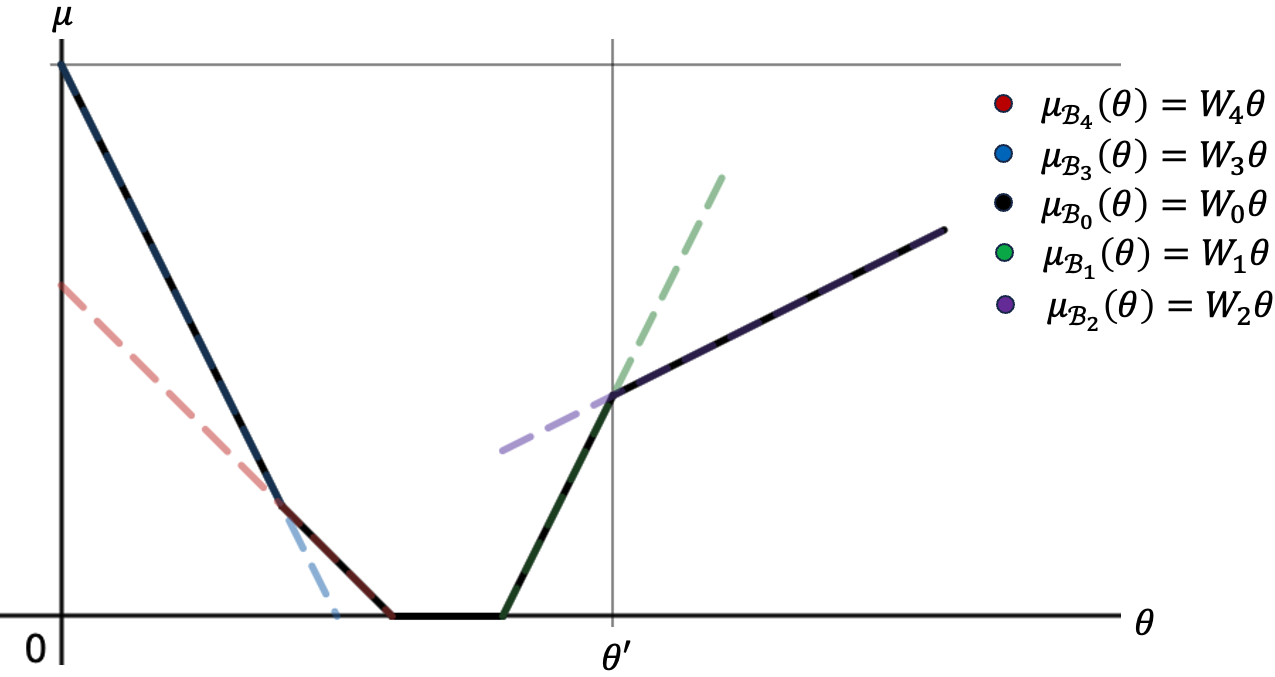}
		\end{minipage}\hfill
		\begin{minipage}[t]{.45\linewidth}
			\centering
			\includegraphics[width=.6\linewidth]{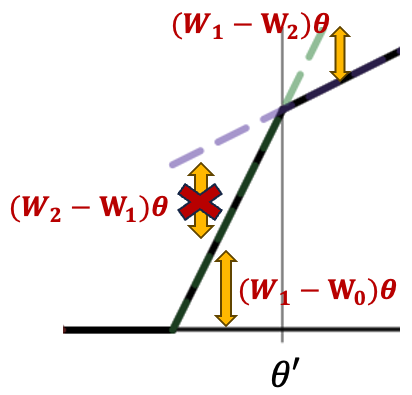}
		\end{minipage}
		\caption{Illustration of slope update rules on an example PWL function.}
		\label{fig:mu_example}
	\end{figure}
	
	The proposed model, $f_\mu$, calculates the shadow prices in multiple layers, by first calculating an array of candidate shadow prices, $[\bfm\mu^*_{\mathcal{B}_1},\dots,\bfm\mu^*_{\mathcal{B}_k}]$, in the first layer and then selecting the correct one, $\bfm\mu^*$ from this array in later layers. To show the motivation behind the architecture, consider the 2-dimensional PWL function on the left in Figure \ref{fig:mu_example}, consisting of five linear segments, each expressed as a linear function, $\mu_{\mathcal{B}_i}(\theta) = \textbf{W}_i\bfm \theta$. Each linear segment is adjacent to one or two other segments, and the two functions are equal at the intersecting points, $\theta''$, between adjacent segments, $i$ and $j$; i.e., $\mu_{\mathcal{B}_i}(\theta') = \mu_{\mathcal{B}_j}(\theta')$ and $\mu_{\mathcal{B}_i}(\theta'') = \mu_{\mathcal{B}_j}(\theta'')$. Note that in the case of a multidimensional PWL function, linear segments are multidimensional, can be adjacent to more than two other segments, and adjacent segments intersect at a multidimensional line instead of a point.
	
	To control the transition from one slope to another, our model starts with a base slope ($\textbf{W}_0$ in the example), and adds the incremental change between two slopes. As shown on the right side of Figure \ref{fig:mu_example}, the slope increases from $W_0$ to $W_1$,  and then decreases from $W_1$ to $W_2$. This function can be represented as
	\begin{align}
		f(\theta) &= [\textbf{W}_0^T\theta]^+ + [(\textbf{W}_1-\textbf{W}_0)^T\theta]^+- [(\textbf{W}_1-\textbf{W}_2)^T\theta]^+,
	\end{align}
	where $[z]^+$ is the ReLU activation function, which returns 0 if $z\leq 0$ and $z$ if $z>0$. Therefore, 
	\begin{align}
		f(\theta) &= 
		\begin{cases}
			\textbf{W}_0^T \theta  & \textbf{ if } \theta<\theta',\\
			\textbf{W}_0^T\theta  +(\textbf{W}_1-\textbf{W}_0)^T\theta   & \textbf{ if } \theta' \leq \theta<\theta'',\notag\\
			\textbf{W}_0^T\theta  +(\textbf{W}_1-\textbf{W}_0)^T \theta  -(\textbf{W}_1-\textbf{W}_2)^T \theta & \textbf{ if } \theta \geq \theta'',\\
		\end{cases} \\&=\begin{cases}
			\textbf{W}_0^T \theta  & \textbf{ if } \theta<\theta',\\
			\textbf{W}_1^T\theta   & \textbf{ if } \theta' \leq \theta<\theta'',\\
			\textbf{W}_2^T\theta   & \textbf{ if } \theta \geq \theta''.
		\end{cases} 
	\end{align}
	In the above example, the '+' or '-' signs between terms and in front of each slope difference depend on whether the slope increases or decreases from one critical region to the next. The rule is as follows: When transitioning from $CR_i$ to $CR_j$, the slope must update from $W_i$ to $W_j$. For $\theta_j$ sampled from the new critical region, $CR_j$, if $\textbf{W}_j\theta_j>\textbf{W}_i\theta_j$, then one can conclude that the slope is updated upwards and downwards otherwise. Therefore, the incremental change in slope is
	\begin{align}
		\Delta f(\theta) &= 
		\begin{cases}
			+(\textbf{W}_j-\textbf{W}_i)\theta   & \textbf{ if } (\textbf{W}_j-\textbf{W}_i)\theta_j>0,\\
			-(\textbf{W}_i-\textbf{W}_j)\theta   & \textbf{ if } (\textbf{W}_j-\textbf{W}_i)\theta_j<0.
		\end{cases} 
	\end{align}

    Using these properties, we propose a NN model to estimate the shadow price function, $\mu(\bfm\theta)$.
    
	\subsection{Shadow Price Model} \label{sec:shadowPriceModel}
	The shadow prices of the inequality constraints are predicted via a deep NN model with three hidden layers, $f_\mu(\cdot)$, as showed in Figure \ref{fig:mu_model}, whose general form is detailed in Appendix \ref{sec:NNs}. Briefly, the first layer weights are fixed to the slopes of linear segments of the shadow price function that are obtained by following the calculations in (\ref{eq:linearinverse_ineq}), and the second and third layer weights are designed to choose the correct slope for a given input. Mathematically, the model can be written as
	\begin{equation}
		f_\mu (\theta) = \textbf{W}^{2T} \sigma(\textbf{W}^{1T}\sigma(\textbf{W}^{0T}(-\textbf{B} - \bfm\theta))).
	\end{equation}
	We assume there are $k$ critical regions in the global solution function and initially assume all regions and corresponding $\mathcal{B}_i$ are known (we later relax this assumption with the introduction of the LvD algorithm). If $\mathcal{B}_i$ are known, then $\bfm\mu^*_{\mathcal{B}_i}$ can be calculated following the operations in (\ref{eq:linearinverse_ineq}). The Shadow Price Model,  $f_\mu(z)$, fixes the first layer weights by stacking all $\bfm\mu^*_{\mathcal{B}_i}$ vertically as
	\begin{equation}
		\textbf{W}^0 = [\nabla\bfm\mu^*_{\mathcal{B}_1},\nabla\bfm\mu^*_{\mathcal{B}_2},\dots,\nabla\bfm\mu^*_{\mathcal{B}_k}]^T,
	\end{equation}
	where $\nabla \bfm\mu^*_{\mathcal{B}_i} \in \mathbb{R}^{m_2\times d}$ and $d=n+m_1+m_2$. Notice that the size of $\nabla \bfm\mu^*_{\mathcal{B}_i}$ is fixed assuming all inequality constraints bind in all cases. On the other hand, the size of $\nabla \bfm\mu^*_{\mathcal{B}_i} \in \mathbb{R}^{r\times c}$ calculated using (\ref{eq:linearinverse_ineq}) changes with the number of binding constraints, i.e.,  where $r\leq m_2$ and $c\leq d$. This is to standardize the input of the NN model and is done by appending vectors of zeros as slopes of nonbinding constraints. In other words, when a constraint $j$ does not bind, the corresponding row of $-\textbf{B}-\bfm\theta$ is multiplied with zero to calculate $\mu_j^*=0$.
	
	Assuming the known values of slopes of $\mu(\bfm\theta)$ in all critical regions, the first layer output will produce all candidate optimal shadow prices as
	\begin{equation}\label{eq:h1}
		\textbf{h}^1 = \textbf{W}^0(-\textbf{B}-\bfm\theta) = [\bfm\mu^*_{\mathcal{B}_1},\bfm\mu^*_{\mathcal{B}_2},\dots,\bfm\mu^*_{\mathcal{B}_k}]^T.
	\end{equation}

    \begin{figure}
		\centering
		\includegraphics[scale=.18]{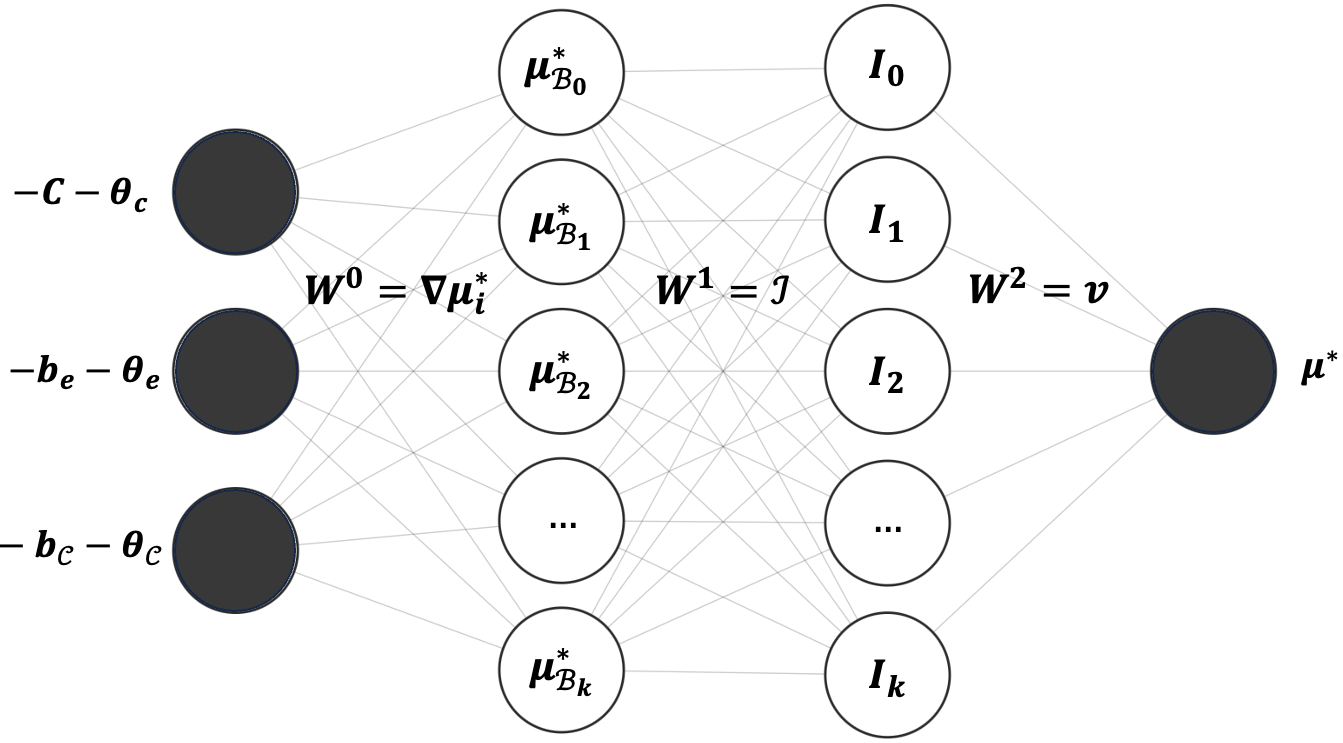}
		\caption{$f_\mu(\bfm\theta)$ model}
		\label{fig:mu_model}
	\end{figure}
	
	\subsubsection{Incidence Matrix and Direction Vector}
    To choose the correct slope in each critical region and form the PWL shadow price function, we linearly combine the first layer outputs by adding one to the other, or subtracting one from another. These operations are handled by two additional NN layers, whose weights are defined as Incidence Matrix, $\mathcal{I} \in \mathbb{R}^{n \times n}$, and Direction Vector, $v\in \mathbb{R}^{n \times 1}$. The former is a sparse matrix that forms the increments from the first layer outputs in (\ref{eq:h1}). For adjacent $(i,j)$,
	\begin{align}
		(\mathcal{I}_{ij}, \mathcal{I}_{jj}) =\begin{cases}
			(-1,1) & \text{ if } (\textbf{W}_j-\textbf{W}_i)\theta_j>0, \\
			(1,-1)  & \text{ if } (\textbf{W}_j-\textbf{W}_i)\theta_j<0,\\
			(0,0)  & otherwise.
		\end{cases} 
	\end{align}
	Similarly, the direction vector can be written as
	\begin{align}
		v = \begin{cases}
			1  & \textbf{ if } (\textbf{W}_j-\textbf{W}_i)\theta_j>0,\\
			-1  & \textbf{ if } (\textbf{W}_j-\textbf{W}_i)\theta_j<0.
		\end{cases} 
	\end{align}

    If the slope increases when moving from $CR_i$ to $CR_j$, the size of the incidence matrix expands with a new row and column and the direction vector expands one dimension. $\mathcal{I}[i,j]$ takes the value of $(-1,1)$ and direction vector expands with $v[j] = 1$ to form $+(\textbf{W}_j-\textbf{W}_i)^+$ function. On the other hand, if the slope decreases, then $\mathcal{I}[i,j]$ takes the value of $(1,-1)$ and direction vector expands with $v[j]=-1$ to form $-(\textbf{W}_i-\textbf{W}_j)^+$ function.

    Using $\mathcal{I}$, $v$, and the $\bfm\mu^*$ prediction, the second and final layer outputs are calculated as
    \begin{subequations}
    \begin{align}
         \textbf{I} &= \sigma(\mathcal{I}^T\textbf{h}^1),\\
         \bfm\mu^* &= v^T\textbf{I}.
    \end{align}
    \end{subequations}
    The second layer, $\textbf{h}^2$, consists of differences between shadow prices for each adjacent critical regions $i$ and $j$, i.e. $\bfm\mu_i^+-\bfm \mu_j^+$. The final layer sums all these activated increments to calculate $\bfm\mu^*$.
    

	\subsection{Learning via Discovery Algorithm}
	\label{sec:algorithm}
	The Shadow Price Model, $f_\mu(\cdot)$, is constructed by using the slopes of the linear segments in each critical region, $\nabla\bfm\mu^*_{\mathcal{B}_i}$, along with the incidence matrix and direction vectors that are calculated using these slopes. The Learning via Discovery  algorithm employs an iterative process to obtain model parameters without the need for training. The algorithm starts at an initial feasible point, $\bfm\theta^0$,  with the solution to the problem at this point, and a search pattern, $\mathcal{P}=[\mathcal{D}_1,\dots\mathcal{D}_n]$, consisting of sets of input parameters, $\mathcal{D}_i = [\bfm\theta_i^0,\bfm\theta_i^1,\dots,\bfm\theta_i^d]$, constructed by incrementally increasing the initial parameters in a specified direction. Using the initial solution, a NN model, $f(\bfm\theta)$, is initialized with the slope of the solution function, $\nabla\bfm\mu^*_0$, in the critical region corresponding to the set of binding constraints, $\mathcal{B}_0$. Initially, the model assumes that there is only one critical region but is already capable of predicting solutions to all parameters within this region, $\bfm\theta \in CR_0$. However, solutions calculated by the model for input parameters in $\mathcal{D}_i$ where $\bfm\theta\not \in CR_0$ will violate KKT optimality conditions. If predictions are found to violate KKT conditions for any input, we can infer that a new critical region has been discovered, given that the input parameters are sampled from the feasible domain. The NN model is then expanded by adding the slope of the new critical region and the process is repeated until all critical regions are discovered. Algorithm \ref{alg:learning} in Appendix \ref{sec:LvDalgorithm} details our approach and an illustration on a 2D example is provided in Appendix \ref{sec:illustration}.
    
	When a new critical region is discovered, the set of binding constraints in this region can be found without solving the problem using the properties described in Theorem \ref{th:transition}. That is, when it is found that the model cannot satisfy KKT conditions, the new set of binding constraints can be found by expanding the set with the violated constraint or dropping the constraint for which $\nabla\bfm\mu^*_{\mathcal{B}_{i-1}} (-\textbf{B}-\bfm\theta_{i}) < 0$. The following remark summarizes the rule to obtain the set of binding constraints in the newly discovered critical region.
    \begin{tcolorbox}[title=Remark 1]
  When $\bfm\theta$ is moved from a critical region to an adjacent one, one of the following can occur: 
\begin{enumerate}
		\item A new inequality constraint, $k$, starts to be violated, i.e., $\textbf{b}_{k} + \bfm\theta_{k} - \textbf{A}_{k}x(\bfm\theta) >0$. Then the set of binding constraints in the new critical region, $\mathcal{B}_{j+1} = \mathcal{B}_{j} \cup \{k\}$.
		\item One of the constraints that was binding in the previous region is no longer binding i.e., $\mu_k^*(\bfm\theta_{j-1})>0 \ \& \ \mu_k^*(\bfm\theta_{j}) = 0$. Then $\mathcal{B}_{j+1} = \mathcal{B}_{j} \setminus \{k\}$.
	\end{enumerate}
\end{tcolorbox}    
	

	\begin{table}[!h]
	\centering
	\begin{tabular}{cccccc}
		KKT1 & KKT2 $(=)$ & KKT2 $(\leq)$ & KKT3 & KKT4 \\
		\midrule
		$\left(\frac{\partial L}{\partial \textbf{x}}\right)^2$ &
		$\left(\frac{\partial L}{\partial \boldsymbol{\lambda}}\right)^2$ &
		$\left[\max\left(0,\frac{\partial L}{\partial \boldsymbol{\mu}}\right)\right]^2$ &
		$\max(0,-\bfm\mu)$ &
		$\left(\boldsymbol{\mu}^*\frac{\partial L}{\partial \boldsymbol{\mu}}\right)^2$
	\end{tabular}
	\caption{KKT Conditions and corresponding formulae}
	\label{tab:kkt}
\end{table}
	To test the KKT optimality of each calculated solution, the algorithm computes a custom $KKT(f(-\textbf{B} - \bfm\theta),\bfm\theta) $ function defined as the average of the five KKT violations formulated in Table \ref{tab:kkt}. Specifically, $KKT(\cdot)$ averages the violations of Stationarity (KKT1), Primal Feasibility of equality constraints (KKT2 $=$), Primal feasibility of inequality constraint (KKT2 $\leq$), Dual Feasibility (KKT3), and Complementary Slackness (KKT4) conditions. The violations in Table \ref{tab:kkt} are vectors, so $KKT(\cdot)$ function first stacks all the vectors, then averages the elements to return a scalar. The function yields 0 for optimal and feasible solutions and nonzero otherwise. 

The outcome of Algorithm \ref{alg:learning} is highly sensitive to the search pattern, $\mathcal{P}$, whether it visits all critical regions, and whether the transitions from one region to another satisfy the assumptions in Theorem \ref{th:transition}. For the numerical results reported in Section \ref{sec:test_results}, we used two different search patterns based on the different kinds of constraints included in the test problems. The problems described in Sections \ref{sec:test_results} and \ref{sec:results_speed} included only box constraints, which resulted in a simple connectivity graph between critical regions that could be discovered using a relatively simple search pattern, as illustrated in Figure \ref{fig:criticalRegions} in Appendix \ref{sec:line-limits}. However, when line limit (i.e., capacity) constraints were included in the problem, incrementally increasing one parameter while fixing the others led to the discovery of different critical regions, resulting in a more complex connectivity graph. Therefore, we updated our search pattern in a more comprehensive algorithm in Appendix \ref{sec:line-limits}.


    

\section{Numerical Results}\label{sec:results}
To obtain proof of concept support for our proposed modeling approach, we applied the approach to the DC optimal power flow problem. The problem is defined as a quadratic program with linear constraints where the objective is to find the generator dispatch that minimizes the cost of generating electricity while satisfying the power flow constraints and generator dispatch limits:
\begin{subequations}
	\begin{align}
		\mathop{\min}_{\textbf{P}_g,\bfm\delta} \ &\textbf{P}^{T}\textbf{Q}\textbf{P} + \textbf{C}^T\textbf{P} + \textbf{C}_0,\\
		\text{s.t.: } & \textbf{P}_d +\bfm\theta_e-\textbf{P} -\textbf{B}\bfm\delta = 0,\\
		&\textbf{P}^-\leq \textbf{P} \leq \textbf{P}^+,
	\end{align}
	\label{eq:DC-OPF}
\end{subequations}
where $\textbf{Q}, \textbf{C}, \textbf{C}_0$ are cost coefficients and $\textbf{B}$ is the susceptance matrix, all of which are constant system parameters. $\textbf{P}\in \mathbb{R}^N$ is the vector of generator dispatches with $P_i = 0 \ \forall i \notin \mathcal{N}_g$ , $\delta_i \in \mathbb{R}^N$ is voltage angles at every bus with $\delta_i=0$ for the slack bus, and $\textbf{P}^d \in\mathbb{R}^N$ is the demand at every bus.

We applied the proposed closed-form (CF) NN approach to discover solution functions for DC-OPF problems on the IEEE-6, -30, and -57 bus test systems, as provided by the Pandapower Python package \citep{thurner2018pandapower}. The 6-bus system comprises 8 primal and 12 dual variables (6 equality and 6 inequality constraints); the 30-bus system has 35 primal and 42 dual variables (30 equality, 12 inequality constraints); and the 57-bus system has 63 primal and 71 dual variables (57 equality, 14 inequality constraints). We compared solution predictions against those of a DNN model trained by standard methods, (i.e., employing input-output pairs that link particular configurations of system parameters to their corresponding optimal solutions), and solutions produced by the Gurobi solver (v10.0.3). For each test system, Algorithm \ref{alg:learning} was used to derive the CF NN model. The generic DNN model was trained and validated on datasets of 5000 and 1000 observations, respectively, and all models were compared on 4000 test observations. 

We used two procedures to generate test datasets, referred to as local perturbations and extreme characteristics, to compare model performance on demand data from inside and outside the DNN training distribution, respectively. The local perturbations data were generated by multiplying base demand data in the Pandapower package with a random uniform number, i.e., $P_{di} = P_{di}^{base} \times Uniform(0.6,1.4) \ \forall i \in \mathcal{N}$. To generate the extreme characteristics data, demand for one bus was increased from 0 to the sum of all generator upper limits, while setting all other loads to 0.01, and repeating this procedure for every bus. For both datasets, solutions were obtained using the Gurobi solver and any infeasible values were labeled.  Experiments and training were run on a Mac Mini M1 (2020) using Python 3.9.16 and PyTorch 2.0.0.post3, with either the default 32-bit floating point precision, which allows for predictions accurate to 7 decimal places, or 64-bit precision, which allows for predictions accurate to 16 decimal places. Model performance was compared based on the KKT violations in Table \ref{tab:kkt}.

\subsection{Test Results}\label{sec:test_results}

Table \ref{tab:MSE_local} reports average KKT violations by our models (CF), DNN with standard training   and Gurobi on the local perturbations test data, with CF average and worst-case (i.e., maximum squared error) results calculated at both 32-bit (CF32) and 64-bit (CF64) representation. Stationarity violations (KKT 1) for the two primal variables are reported separately as $KKT1-\textbf{P}_g$ and $KKT1-\delta$. The CF32 model produced results competitive with Gurobi and calculated solutions with very small error margins. For 6- and 30-bus systems, KKT violations were less than 4.45E-10 for all measures. For the 57-bus system, average violations increased to a maximum of 8.41E-08 and the worst case single violation was 2.42E-05. The CF32 model outperformed the DNN model for all cases and all measures, and even its worst case predictions exceeded average DNN performance in almost all cases. This is remarkable considering that the test dataset was sampled from the same distribution as the DNN training set.

The CF32 model was expected to outperform DNN, but given that it represents the closed-form global solution function exactly, it might also be expected to outperform a numeric solver like Gurobi, which estimates solutions iteratively. To test if performance of the CF32 model was limited by the default 32-bit floating point precision of the NN package (PyTorch), we increased precision to 64-bits and recalculated model predictions. As shown in Table \ref{tab:MSE_local}, average errors decreased from $\leq 8.41$ E-08 for CF32 to $\leq 3.15$E-25 for CF64, substantially outperforming Gurobi, and even the worst-case CF64 violation observed (1.17E-22) exceeded Gurobi's average performance overall. Results for the extreme characteristics test dataset are reported in Appendix \ref{sec:further_results}, showing similar performance differences between the CF, DNN, and Gurobi models. Whereas DNN performance declined significantly on extreme out-of-distribution test data, the CF32 model remained competitive with Gurobi and the CF64 model outperformed Gurobi on most measures. 


Table \ref{tab:results_cost} in Appendix \ref{sec:further_results} compares the optimal costs computed by Gurobi, \(C(\mathbf{P}^g)\), with those estimated by the CF models, \(\hat{C}(\mathbf{P}^g)\). For both local and extreme demands the CF64 model yields lower costs for the majority of predictions, while for the remaining instances, the difference is on the order of \(10^{-8}\). These results, along with the observed KKT violations, demonstrate that our model achieves lower costs without compromising feasibility or optimality. The effects of including line limit constraints are considered in Appendix \ref{sec:line-limits}, with an updated search pattern to accommodate more complex connectivity between critical regions and numerical results for the 6-bus test case.

\begin{table}
	\centering
	\begin{tabular}{lccccccc}
		\toprule
		&	&KKT1-$\textbf{P}_g$ & KKT1-$\bfm\delta$ & KKT2 $(=)$ & KKT2 $(\leq)$ & KKT3 & KKT4\\
		\midrule
		&  \ 6bus   &  1.69E-30 & 1.17E-27 & 5.11E-30 & 5.78E-36 & 0.00E+00 & 3.80E-31 \\
		CF64 &  30bus  & 5.36E-31 & 5.11E-28 & 1.03E-30 & 0.00E+00 & 0.00E+00 & 0.00E+00 \\ 
		&  57bus  &  2.27E-28 & 3.15E-25 & 2.26E-29 & 4.54E-29 & 0.00E+00 & 7.63E-26 \\
		\midrule
		&  \ 6bus  & 1.26E-29 & 2.05E-26 & 3.08E-29 & 1.11E-31 & 0.00E+00 & 1.34E-29 \\ 
		CF64 &  30bus  & 9.66E-30 & 4.79E-26 & 3.18E-29 & 0.00E+00 & 0.00E+00 & 0.00E+00 \\ 
		(worst)&  57bus  & 8.53E-27 & 4.14E-23 & 2.79E-27 & 5.98E-26 & 0.00E+00 & 1.17E-22 \\
		\midrule
		&  \ 6bus  & 1.12E-12 & 4.45E-10 & 1.46E-13 & 5.64E-11 & 0.00E+00 & 5.16E-15 \\ 
		CF32 &  30bus  & 1.78E-13 & 1.13E-10 & 8.14E-14 & 0.00E+00 & 0.00E+00 & 0.00E+00 \\ 
		&  57bus & 6.48E-11 & 8.41E-08 & 2.93E-12 & 1.91E-11 & 0.00E+00 & 2.11E-08 \\ 
		\midrule
		&  \ 6bus  & 5.09E-12 & 3.44E-09 & 2.47E-12 & 1.20E-03 & 0.00E+00 & 5.62E-13 \\ 
		CF32 &  30bus  & 2.79E-12 & 1.03E-08 & 9.88E-12 & 0.00E+00 & 0.00E+00 & 0.00E+00 \\ 
		(worst)&  57bus  & 2.90E-09 & 1.46E-05 & 1.83E-10 & 4.55E-08 & 0.00E+00 & 2.42E-05 \\ 
		\midrule
		& \ 6bus & 2.72E-06 & 7.04E-06 & 2.93E-05 & 9.70E-08 & 1.36E-08 & 1.01E-07 \\
		DNN & 30bus & 4.08E-07 & 1.86E-06 & 5.84E-05 & 0.00E+00 & 2.17E-09 & 6.00E-09 \\
		& 57bus & 2.44E-04 & 5.90E-05 & 5.21E-03 & 1.19E-08 & 3.27E-08 & 2.58E-06 \\
		\midrule
		&  \ 6bus & 1.52E-15 & 9.78E-28 & 2.64E-32 & 0.00E+00 & 0.00E+00 & 2.37E-16 \\
		Gurobi & 30bus & 1.82E-16 & 3.32E-21 & 6.24E-19 & 0.00E+00 & 0.00E+00 & 1.08E-19\\
		&  57bus  & 4.01E-13 & 8.63E-10 & 5.96E-15 & 0.00E+00 & 0.00E+00 &  5.13E-19 \\
		\bottomrule
	\end{tabular}
	\caption{Mean Squared KKT Errors on the test sets with Local Perturbations}
	\label{tab:MSE_local}	
\end{table}


\subsection{Model Performance on Uncertain Inputs}\label{sec:results_speed}

The CF NN model provides significant computational advantages over analytic solvers like Gurobi. To evaluate the performance of our model relative to Gurobi, we created 1000 random sets of input parameters and solved the associated problems for each test system. For the 6-bus system, Gurobi required 14.6 seconds to solve all 1000 instances, which increased to 22.5s and 43.6s for 30- and 57-bus systems, respectively. The corresponding solution times for the CF32 model were 0.003s, 0.004s, and 0.006s, which increased slightly to 0.003s, 0.006s, and 0.010s, respectively, for CF64. Thus, solution times were much longer for Gurobi and increased at a faster rate with increasing system size.
The CF model’s computational efficiency and robust KKT results suggest that it is capable of rapidly producing extensive sets of optimal and feasible solutions, which are valuable for simulating and planning energy systems with uncertain resources like renewable wind generation. To investigate, we used the CF model to simulate energy dispatches on the 57-bus system with uncertain renewable resources over 1-day and 1-year planning horizons. To generate input parameters, we used default loads from the Pandapower package as base demand; historical hourly demand data for Ontario, Canada were then scaled down within the range \{0, 1\} to obtain an hour constant. The demand at each bus was scaled by the hour constant to produce the corresponding hourly input values. We selected May 11, 2024 for 1-day planning and 2024 (January 1st to December 31st) for 1-year planning. 

Additionally, we sampled 500 random renewable generator dispatch values per hour, $\textbf{P}_{ren}$, from an exponential distribution with $\lambda = 1.25$. A maximum generator output of 1.5 units was enforced by capping any generated values exceeding 1.5. The generator dispatch was injected to the test system by modifying the DC-OPF power balance equation (\ref{eq:DC-OPF}b) with a nondispatchable renewable generator as follows:
\begin{equation}
	\textbf{B}\boldsymbol{\delta}-\textbf{P}_g-\textbf{P}_{ren} + \textbf{P}_d=0.
\end{equation}

The results are shown in Table \ref{tab:uncertain_57bus}. The CF32 and CF64 models were able to predict the optimal solutions with very small errors, for one day in 16 and 31 milliseconds, and for one year of data in 6.23 and 14.6 seconds, respectively. Comparable solver-based computation times for Gurobi are estimated to be 8.7 minutes for daily data and 53.2 hours for annual data, respectively, based on the time needed for Gurobi to solve 1000 57-bus problems reported above.  


\begin{table}[!h]
    \centering
    \resizebox{\textwidth}{!}{
    \begin{tabular}{c|c|c|c|c|c|c|c|c|c}
    \toprule
 & & KKT1 ($\textbf{P}_g$)  & KKT1 ($\bfm\delta$)  & KKT2 $(=)$  & KKT2 $(\leq)$  & KKT3  & KKT4  & $\tau$ & $\tau$(Gurobi) \\
 \midrule
\multirow{2}{*}{1 day} & CF64 & 3.74E-28 & 5.37E-25 & 4.67E-29 & 3.36E-28 & 0.00E+00 & 8.20E-25 &  0.031s  &  \multirow{2}{*}{$\approx$8.7 min}  \\
 & CF32  & 1.08E-10 & 1.44E-07 & 3.83E-12 & 4.66E-11 & 0.00E+00 & 8.87E-08 &  0.016s  &   \\
\hline
 \multirow{2}{*}{1 year} & CF64  & 2.07E-28 & 2.86E-25 & 2.38E-29 & 9.90E-20 & 0.00E+00 & 1.18E-26 &  14.6s   & \multirow{2}{*}{$\approx$53.2 hr}   \\
 & CF32  & 5.98E-11 & 8.06E-08 & 2.63E-12 & 1.82E-12 & 0.00E+00 & 2.28E-09 &  6.23s   &    \\
 
 \bottomrule
    \end{tabular}}
    \caption{Performance Comparison of Closed-form and Gurobi on 57-bus system to calculate solutions to long-term planning problem}
	\label{tab:uncertain_57bus}
\end{table}

\section{Conclusion}\label{sec:conclusion}

This paper proposed a NN-based approach to discover the closed-form solution to QP with linear constraints. The method uses algebraic operations to derive all NN model weights from the underlying problem coefficients analytically without requiring training sets. Instead of training to reduce MSE error between model predictions and true values, the CF NN model learns the slopes corresponding to each critical region of the true solution function analytically, and thereby guarantees the optimality and feasibility of solutions for every critical region discovered. 

The proposed model has significant implications. First, the model implements the end-to-end QP solution procedure in an entirely explainable way, by using two networks to sequentially calculate shadow prices of inequality constraints and then construct the rest of the solution. Compared to black box DNN models that have difficulty generalizing outside the training dataset, the CF NN approach generalizes seamlessly to guarantee optimality over all discovered critical regions. Compared to multiparametric methods that apply specific optimizer rules for each CR, the CF NN model reflects the global solution function over all discovered CRs, so it avoids the point-location problem as pre-processing to identify the CR and apply the correct rule is not required. Moreover, because the NN model directly represents the solution function, it can exceed the precision of iterative analytic solvers like Gurobi, at far less computational cost.  

There are limitations to our approach in its current form. First, the search pattern employed by the discovery algorithm does not generalize to all cases, and different patterns were used in the study for each application. In general, our NN modeling approach is agnostic to the search pattern employed, and future work is needed to identify patterns that generalize to a wide variety of cases. Second, our learning algorithm relies on Theorem \ref{th:transition}, which posits that only one constraint either starts or stops binding when moving from one critical region to another, ignoring more complex cases such as degeneracy. Future work will address this limitation to improve generalizability. Third, the proposed algorithms derive model weights using matrix inversion, which can be computationally costly for large systems. Future work will consider more efficient approaches to obtain model weights. Nonetheless, the LvD algorithm has efficiency advantages over existing search methods. As the NN model expands with each discovered critical region, its current state is immediately used to predict solutions in search of the next region, rather than relying on analytic solvers.    

Finally, our study leaves two important classes of QP problems for future research. The problems we solved allowed the RHS of equality constraints to vary but fixed the inequality constraints. Varying inequality constraints result in more complex connectivity graphs, which demand careful examination and potentially updating the learning algorithm. Additionally, we did not consider integer variables, which can result in nonlinear boundaries between critical regions. Future work is needed to address these cases and provide complete parametric solutions to mp-QP problems.

\section*{Acknowledgments}
Sections of this paper are adapted from the author’s PhD dissertation completed at University of Waterloo.

\bibliographystyle{informs2014} 
\bibliography{refs}

\appendix


\section{Additional Background}

\subsection{Proof of Piecewise Linearity of the Solution Function}\label{sec:pwl_proof}
\begin{proof}
	\ Consider the inequality constrained problem in (\ref{eq:ineq}), $ICP$, and a subproblem that does not contain constraint, $k$, $ICP_s$. Let $CR_i, CR_j\subset \Theta_f$ be two arbitrary adjacent critical regions, so that for any $\bfm\theta_i \in CR_i, \bfm\theta_j\in CR_j$, there exists $w$ and $\bfm\theta = w\bfm\theta_i + (1-w)\bfm\theta_j$ such that $\bfm\theta^* \in CR_i \cup CR_j$. Without loss of generality, assume for the set of binding constraints in $CR_i$ and $CR_j$, $\mathcal{B}_j \setminus \mathcal{B}_i = \{k\}$ holds. 

    Consider the following solution functions:
    \begin{align}
		f_i(\bfm\theta) =  
		\textbf{J}^{-1}_{\mathcal{B}_i}
		\begin{bmatrix}
			-\textbf{C} \\ -\textbf{b}_e-\bfm\theta\\ -\textbf{b}_{\mathcal{B}_i}\end{bmatrix}= \begin{bmatrix}
			x(\bfm\theta)\\ \lambda(\bfm\theta) \\ \mu_{\mathcal{B}_i}(\bfm\theta)
		\end{bmatrix},
        \quad
		f_j(\bfm\theta) =  
		\textbf{J}^{-1}_{\mathcal{B}_j}
		\begin{bmatrix}
			-\textbf{C}\\ -\textbf{b}_e-\bfm\theta\\ \textbf{b}_{\mathcal{B}_i} \\ -\textbf{b}_{\{k\}}
		\end{bmatrix}=
        \begin{bmatrix}
			x(\bfm\theta)\\ \lambda(\bfm\theta) \\ \mu_{\mathcal{B}_i}(\bfm\theta) \\ \mu_{\{k\}}^*(\bfm\theta)
		\end{bmatrix}.
        \label{eq:ICP_slope_change}
	\end{align}
    
    We will show that optimal solution function, $f(\bfm\theta) = [\textbf{x}^*,\bfm\lambda^*,\bfm\mu^*]$ for any $\bfm\theta \in CR_i \cup CR_j$ is a PWL function in this subset of the feasible region. 
    
    By definition, $f_{i}(\bfm\theta)$ and $f_{j}(\bfm\theta)$ produce optimal solutions to all $\bfm\theta \in CR_i \cup CR_j$, respectively. Therefore, 
    \begin{equation}
        f(\bfm\theta) = 
        \begin{cases}
            f_i(\bfm\theta), \quad \forall \bfm \theta \in CR_i, \\
            f_j(\bfm\theta), \quad \forall \bfm \theta \in CR_i.
        \end{cases}
        \label{eq:PWLsub}
    \end{equation}
    To prove piecewise linearity of $f$, we need to show that for any $\bfm\theta_i$ and $\bfm\theta_j$, there exists $w^* \in [0,1]$ and $\bfm\theta^* = w^*\bfm\theta_i + (1-w^*)\bfm\theta_j$ such that $f_{i}(\bfm\theta^*) = f_{j}(\bfm\theta^*)$.

    Now, given that $\mu_k$ is known, we can rewrite (\ref{eq:ICP_slope_change}b) as 
    \begin{align}
		f_j(\bfm\theta) &=  
		\textbf{J}^{-1}_{\mathcal{B}_i}
		\begin{bmatrix}
			-\textbf{C} + \textbf{A}_{k}^T\mu_{k}(\bfm\theta) \\ -\textbf{b}_e-\bfm\theta\\ \textbf{b}_{\mathcal{B}_i}\end{bmatrix}=f_i(\bfm\theta) + \begin{bmatrix}
			\textbf{A}_{k}^T\mu_{k}(\bfm\theta) \\ 0\\ 0\end{bmatrix}.
    \end{align}
    Since, $\bfm\mu^*_k \geq 0$, there exists a $\bfm\theta^* \in CR_j$ such that $\mu_k(\bfm\theta^*) = 0$. But for $\mu_k (\bfm\theta^*)=0$, $f_j(\bfm\theta^*) = f_i(\bfm\theta^*)$, hence $\bfm\theta^* \in CR_i$ as well. Since $f_i(\bfm\theta)$ and $f_j(\bfm\theta)$ are linear functions, $f(\bfm\theta)$ in (\ref{eq:PWLsub}) is optimal $\forall\bfm\theta \in CR_i$ and $\forall\bfm\theta \in CR_j$, it is continuous in the two arbitrary adjacent critical regions. Since the choice of adjacent regions is arbitrary, when  $f$ is expanded to cover all critical regions, it will preserve the PWL form.

\end{proof}

\subsection{Neural Networks}\label{sec:NNs}
Neural Networks are universal function approximators that can represent any continuous function, defined as
$f: \mathbf{x} \rightarrow \mathbf{y}$. The simplest form—a shallow feedforward neural network—features an input layer, a hidden layer \( h \), and an output layer. The value (or input current) at any neuron is calculated by multiplying the previous layer's values by a weight matrix, \( \mathbf{W} \). Instead of directly using this input current, the neuron's value is typically modified by an activation function. For example, with a threshold-based activation like ReLU, the neuron activates only if the input current is positive.

\begin{equation}
	ReLU(z_i) = \begin{cases}
		z_i& \text{ if } z_i\geq 0,\\
		0 & \text{ otherwise, }
	\end{cases}
\end{equation}
where $ \sigma(z_i)=ReLU(z_i) $ is an element-wise activation function.

		

The general form of a NN model can be mathematically defined as follows:
\begin{align} \label{eq:BG_deepNN}
	\textbf{h}^0 &= \textbf{x}^T,\notag \\
	\textbf{h}^{i} &= \sigma(\textbf{W}^{i-1T} \textbf{h}^{i-1T}+\textbf{b}^{i-1}), \quad  \text{ for } i \in [1,\dots,n_h],\notag \\
	\textbf{y}  &= \sigma(\textbf{h}^{n_hT}),
\end{align}
where the superscript $T$ denotes the transpose operation and $\textbf{W}^{iT} = (\textbf{W}^{i})^T$ is used to simplify the notation, $ \textbf{x} = [x_1,\dots,x_{d_x}] $, $ \textbf{y} = [y_1,\dots,y_{d_y}] $ are input and output of the model, $ \textbf{h}^i = [h_1^i,\dots,h_{d_h}^i] $ is the calculated values for hidden neurons, $d_x,d_h,d_y$ are dimensions of input, hidden and output layers, $ n_h $ is the number of hidden layers, and $d_d$ is the dataset size (number of observations). Here, output of each consecutive layer is a function of the  previous layer output using weight and bias parameters, $ \textbf{W}^i\in \mathbb{R}^{d_{h} \times d_{h}}$ and $\textbf{b}^i\in \mathbb{R}^{d_h}$, which are optimized during NN training. Note that $d_h$ can differ from one hidden layer to another and denoted with the same term for ease of presentation.

\renewcommand{\theequation}{A.\arabic{equation}}
\setcounter{equation}{0}

\section{Additional Details of Proposed Approach}
\subsection{Derivation of the Solution Function of Inequality Constrained mp-QP}\label{sec:solution}

Assume that the set of binding constraint indices $\mathcal{B}\subseteq \{1,\dots,m_2\}$ is known for an arbitrary parameter set, $\bfm \theta$. Then, the optimal solution to the problem can be found by solving
\begin{subequations}\label{eq:derivativeswithmu}
	\begin{align}
		\frac{\partial L}{\partial \textbf{x}} &= 2\textbf{Q}\textbf{x}+\textbf{C}+\bfm\theta_c - \boldsymbol{\lambda}^T\textbf{A}_e^T-\boldsymbol{\mu}^T\textbf{A}_\mathcal{B}^T = 0\\
		\frac{\partial L}{\partial \boldsymbol \lambda} &= \textbf{b}_e+\bfm\theta_e-\textbf{A}_e\textbf{x} =0,\\
		\frac{\partial L}{\partial \boldsymbol \mu_\mathcal{B}} &= \textbf{b}_{\mathcal{B}}+\bfm\theta_\mathcal{B}-\textbf{A}_{\mathcal{B}}\textbf{x} =0,
	\end{align}
	\label{eq:binding_der}
\end{subequations}
where $\textbf{A}_\mathcal{B}, \textbf{b}_\mathcal{B}$ consist of rows of coefficient matrices, $\textbf{A}_\mathcal{C}, \textbf{b}_\mathcal{C}$, corresponding to coefficient indices in $\mathcal{B} \subseteq \mathcal{C}$, and optimal shadow prices $\mu_i^* > 0 \  \forall i \in \mathcal{B} $ and $\mu_i^* = 0 \  \forall i\notin \mathcal{B} $. 

Note that $\bfm\mu^*$ is a function of the parameter set $\bfm \theta$, which determines the constraints that are active at the solution. 
Given the uncertain parameters, $\bfm\theta_e, \bfm\theta_\mathcal{C}$, our goal is to derive a solution strategy that can calculate optimal solutions for any such parameter. For this purpose, we first assume that the function, $\bfm\mu(\bfm\theta)$ that generates $\bfm\mu^*$, is known. This assumption is later relaxed with the introduction of our proposed NN models. Under this assumption, optimal solutions to (\ref{eq:ineq}) can be  obtained by modifying (\ref{eq:lagr}) and writing the derivatives as follows:
\begin{subequations}
	\begin{align}
		\frac{\partial L}{\partial \textbf{x}} &= 2\textbf{Q}\textbf{x}+\textbf{C}+\bfm\theta_c - \boldsymbol{\lambda}^T\textbf{A}_e^T-\mu^T(\bfm\theta)\textbf{A}_\mathcal{C}^T = 0,\\
		\frac{\partial L}{\partial \boldsymbol \lambda} &= \textbf{b}_e+\bfm\theta_e-\textbf{A}_e\textbf{x} =0.
	\end{align}
	\label{eq:derivatives}
\end{subequations}
The above system of equations can be rearranged in matrix form as follows:
\begin{align}
	\begin{bmatrix}
		2\textbf{Q} & -\textbf{A}_e^T\\
		-\textbf{A}_e & 0
	\end{bmatrix}
	\begin{bmatrix}
		\textbf{x}\\ \boldsymbol{\lambda}
	\end{bmatrix} 
	=
	\begin{bmatrix}
		-\textbf{C}-\bfm\theta_c+\bfm\mu^{T*}\textbf{A}_\mathcal{C}^T \\ -\textbf{b}_e-\bfm\theta_e
	\end{bmatrix},\label{eq:linear}
\end{align}
where $\bfm\mu^*=\mu(\bfm\theta)$. Hence, optimal solutions, $\textbf{x}^*,\bfm\lambda^*$, for any $\bfm\theta\in \Theta_f$ can be calculated as
\begin{align}
	\textbf{J}^{-1}
\begin{bmatrix}
	-\textbf{C}-\bfm\theta_c+\textbf{A}_\mathcal{C}^T\bfm\mu^* \\ -\textbf{b}_e-\bfm\theta_e \qquad 
\end{bmatrix}=\begin{bmatrix}
		\textbf{x}^*\\ \boldsymbol{\lambda}^*
	\end{bmatrix} , \quad \text{ where } \quad
		\textbf{J} = \begin{bmatrix}
				2\textbf{Q} & -\textbf{A}_e^T\\
				-\textbf{A}_e & 0
			\end{bmatrix}.
	\label{eq:linearinverse}
\end{align}

\renewcommand{\theequation}{B.\arabic{equation}}
\setcounter{equation}{0}

\subsection{Learning via Discovery Algorithm} \label{sec:LvDalgorithm}
The learning via discovery algorithm can be found in Algorithm \ref{alg:learning}.
\begin{algorithm}
    \caption{Learning via Discovery}\label{alg:learning}
    \begin{algorithmic}
    \State \textbf{Input:} Search pattern $\mathcal{P} = [\mathcal{D}_1,\dots,\mathcal{D}_n]$, starting point, $\bfm\theta^0$, tolerance $tol$
    \State \textbf{Output:} Learned model parameters
    \State Solve the problem at the initial point $\theta^0$ and find $\mathcal{B}_0$.
    \State Initiate model $f(\cdot)$ with $\textbf{W}^0 = [W^0_{\mathcal{B}_0}], \textbf{W}^1 = [1],\textbf{W}^2 = [1]$ and $\textbf{W}=\textbf{J}^{-1}$
    \For{$\mathcal{D}$ in $\mathcal{P}$}
        \While{$KKT(f(-\textbf{B}-\bfm\theta^i), \bfm\theta^i) \geq tol \quad \forall i \in [1,2,\dots,d]$}
        \State Calculate $f(\bfm\theta^i) \quad \forall i \in [1,2,\dots,d]$
        \State Find smallest $i$ such that $KKT(f(-\textbf{B}-\bfm\theta_i),\bfm\theta_i)> tol$
        \State Find the violated constraint using Remark 1, update $\mathcal{B}_{i+1}$ and $W^0_{i+1}$
        \State $\textbf{W}^0 \gets \textbf{W}^0 + [W^0_{i+1}]$
        \If{$(W^0_{i+1}-W^0_{i})\theta^i>0$}
        \State $(\textbf{W}^1_{ij},\textbf{W}^1_{jj}) \gets (-1,1)$
        \State $\textbf{W}^2_j \gets 1$
        \ElsIf{$(W^0_{i+1}-W^0_{i})\theta^i<0$}
        \State $(\textbf{W}^1_{ij},\textbf{W}^1_{jj}) \gets (1,-1)$
        \State $\textbf{W}^2_j \gets -1$
        \EndIf
        \EndWhile
        \EndFor
        
        \Return $\textbf{W}^0,\textbf{W}^1,\textbf{W}^2$
    \end{algorithmic}
\end{algorithm}

\subsection{Illustration of the Algorithm on a 2D Example}\label{sec:illustration}
We illustrate the functioning of the learning algorithm using the following 2D example from \citet[p. 56]{pistikopoulos2020multi}: 
\begin{subequations}
\begin{align}
    \textbf{Min: } & \sum_{i=1}^4 Q_ix_i^2 + C_ix_i, \\
    \textbf{s.t.:} & \notag \\
    & -x_1-x_3 = \theta_1,\quad -x_2-x_4 = \theta_2,\\
    & x_1+x_2 \leq 350, \quad x_3+x_4 \leq 600,\\
    & x_i \geq 0, \ \bfm\theta_i \in \Theta_f, 
\end{align}
\end{subequations}
where $Q = [156,162,162,126]$ and $C = [25,25,25,25]$ are coefficients of quadratic and linear terms. The feasible set of parameters is $\Theta_f = \{\theta_1,\theta_2\ |\ \theta_1+\theta_2 \leq 1000 \}$.

	Figure \ref{fig:CRs_2D} presents the critical regions of the feasible parameters for this problem geometrically and the corresponding connectivity graph. As shown, there are four critical regions, $CR_1, CR_2,CR_3$, and $CR_4$, defined by the corresponding sets of binding constraints, $\mathcal{B}_1 = \{3,4\}$, $\mathcal{B}_2 = \{1, 3,4\}$, $\mathcal{B}_3 = \{1,3,4,5\}$, and $\mathcal{B}_4 = \{1,3,4,6\}$, respectively. The critical regions 1 and 2, 2 and 3, 2 and 4 are adjacent to each other, respectively.
	\begin{figure}\centering
    \includegraphics[scale=.3]{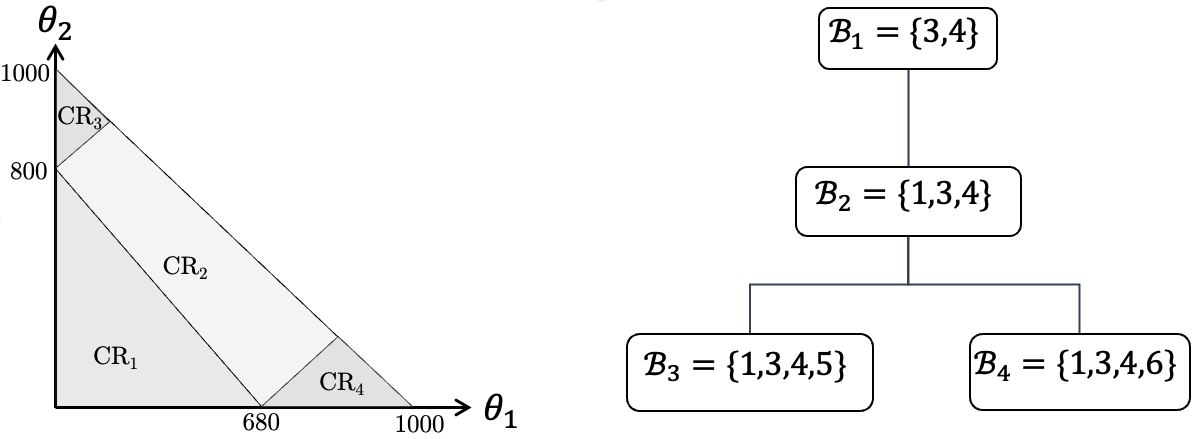}
    \caption{Critical regions and connectivity graph for an mp-QP with two parameters (Adapted from \cite[p. 59]{pistikopoulos2020multi})}
    \label{fig:CRs_2D}
\end{figure}
	
	\begin{figure} \label{fig:algorithm2D}
		\centering
		\begin{minipage}[b]{.5\linewidth}
			\centering
			\includegraphics[width=\linewidth]{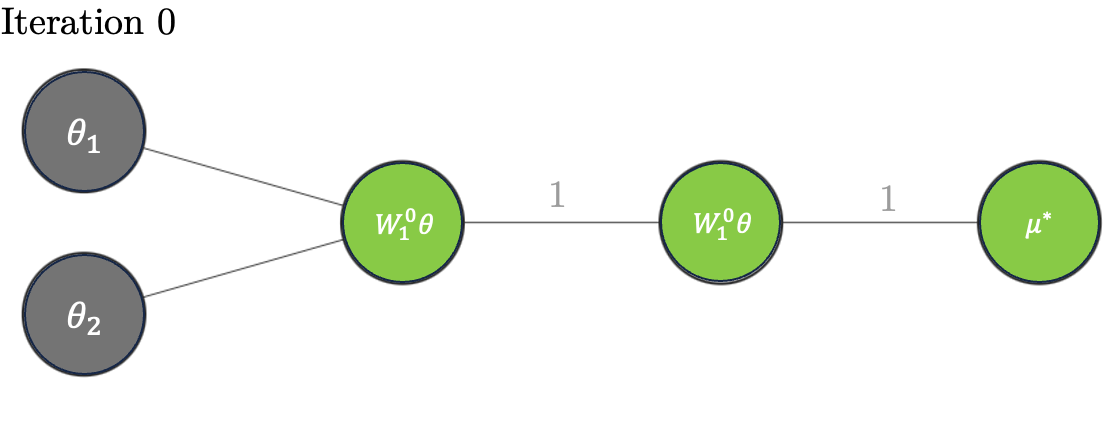}
			\vspace{10pt}
		\end{minipage}
		\begin{minipage}[b]{.3\linewidth}
			\centering
			\includegraphics[width=\linewidth]{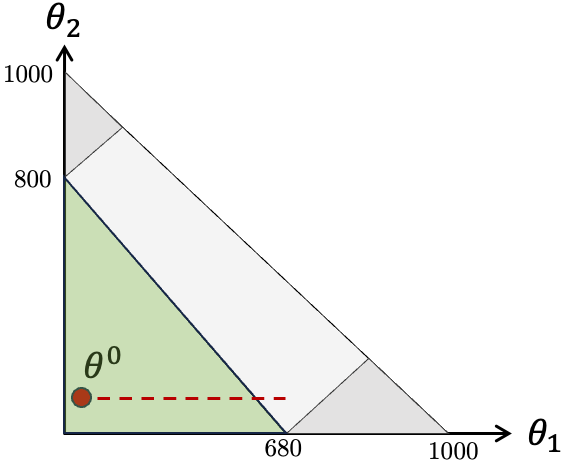}
		\end{minipage}
		
		\begin{minipage}[b]{.5\linewidth}
			\centering
			\includegraphics[width=\linewidth]{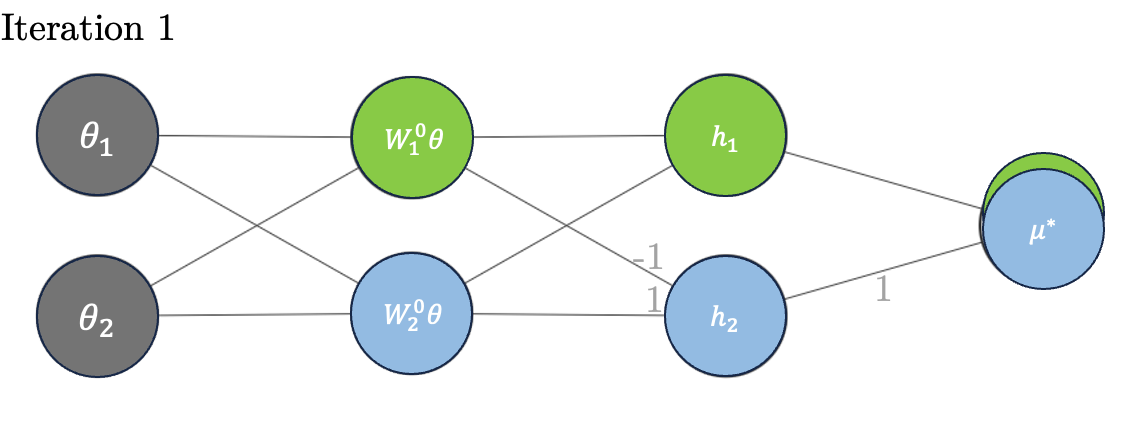}
			\vspace{10pt}
		\end{minipage}
		\begin{minipage}[b]{.3\linewidth}
			\centering
			\includegraphics[width=\linewidth]{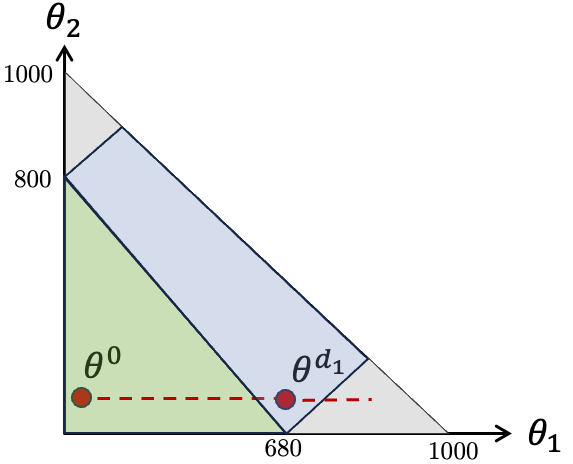}
		\end{minipage}
		
		\begin{minipage}[b]{.5\linewidth}
			\centering
			\includegraphics[width=\linewidth]{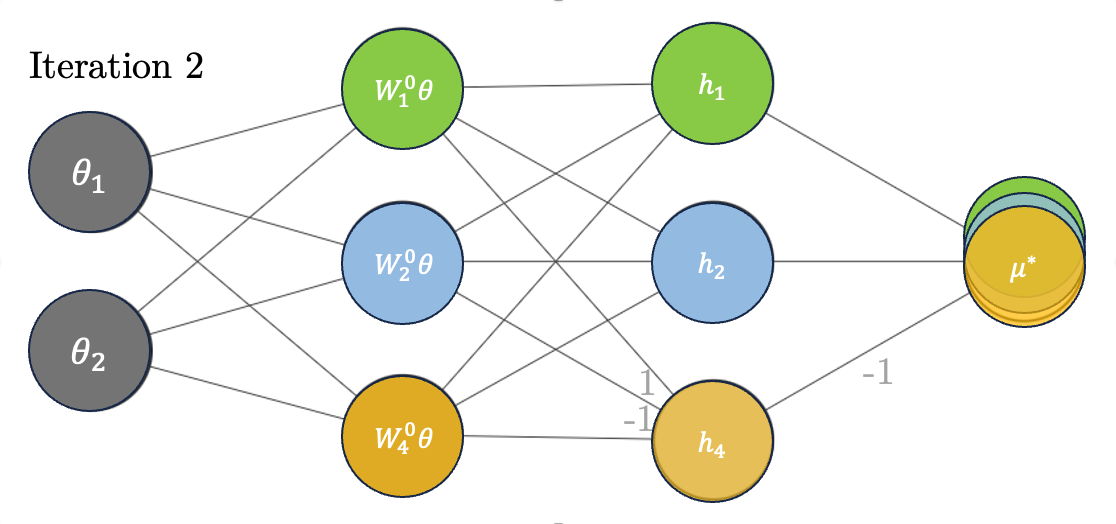}
		\end{minipage}
		\begin{minipage}[b]{.3\linewidth}
			\centering
			\includegraphics[width=\linewidth]{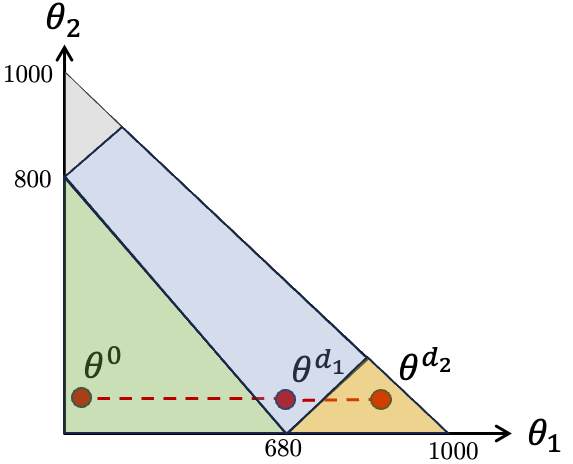}
		\end{minipage}
		
		\begin{minipage}[b]{.5\linewidth}
			\centering
			\includegraphics[width=\linewidth]{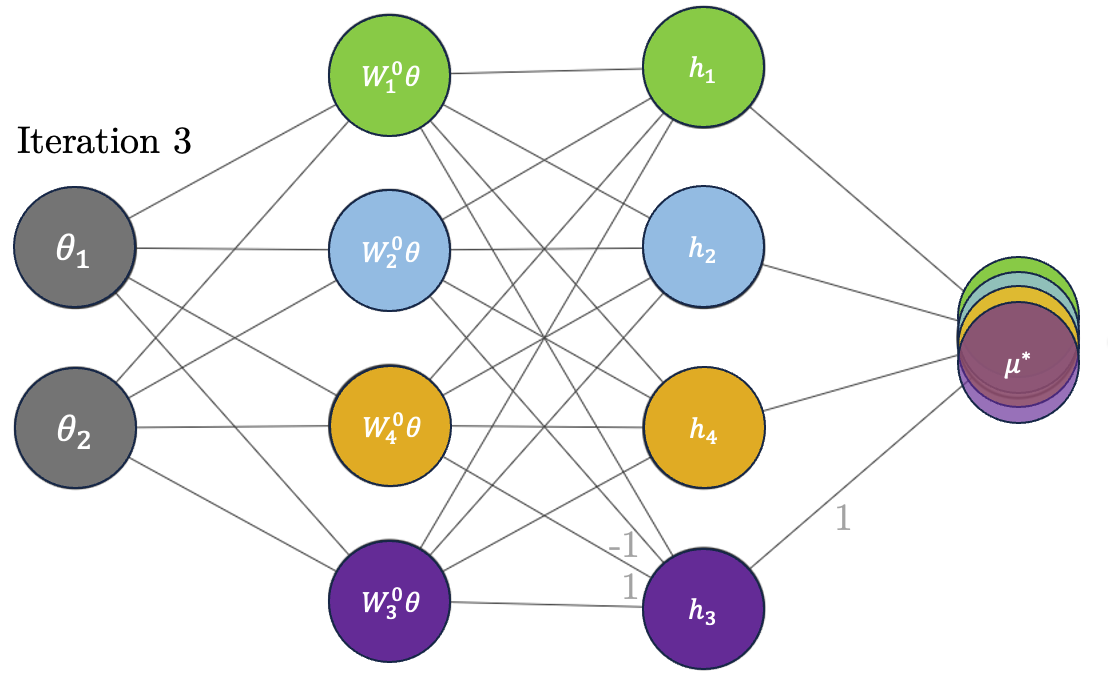}
		\end{minipage}
		\begin{minipage}[b]{.3\linewidth}
			\centering
			\includegraphics[width=\linewidth]{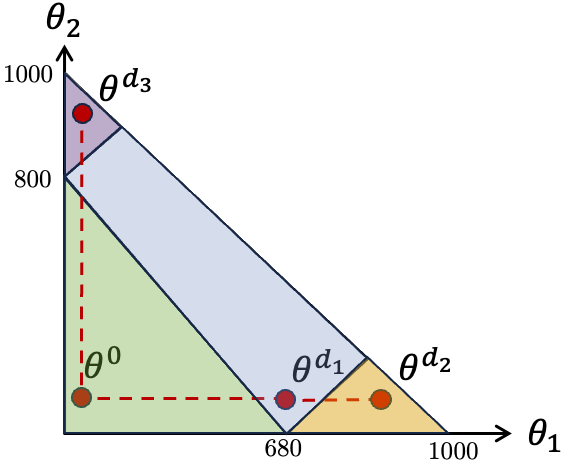}
		\end{minipage}
		
		\caption{Illustration of slope update rules on an example PWL function.}
		\label{fig:CL_algorithm2D}
	\end{figure}
	
	Note that the proposed algorithm does not require information about the critical regions or the connectivity between them prior to the discovery process. As illustrated in Figure \ref{fig:CL_algorithm2D}, the algorithm starts at an arbitrary $\bfm\theta^0$ (iteration 0). As this point is in $CR_1$,  solving the problem at $\bfm\theta^0$ yields $\mathcal{B}_1 = \{3, 4\}$, which is used to calculate $\nabla\bfm\mu_{\mathcal{B}_1}$ and $\textbf{J}^{-1}$ following (\ref{eq:linearinverse_ineq}). Then, the solver model, $f(-\textbf{B} - \bfm\theta;\textbf{W}^0,\textbf{W}^1,\textbf{W}^2,\textbf{W})$ is initialized with $\textbf{W}^0 = [\nabla\bfm\mu_{\mathcal{B}_1}], \textbf{W}^1 = [1], \textbf{W}^2 = 1$, and $\textbf{W} = \textbf{J}^{-1}$. 
	
	As seen in Figure \ref{fig:CL_algorithm2D}, at iteration 0, this initial model can already predict the optimal solutions $\forall \bfm\theta \in CR_1$. This initial model is used to predict solutions for all data points sampled in the chosen search direction. However, as $\bfm{\theta}$ is moved outside the $CR_1$ towards $\bfm{\theta}^{d_1}$, the model starts violating KKT conditions due to the changing set of binding constraints in the new critical region. Then, changes in the set of binding constraints can be identified using Remark 1. Specifically, in the example, the model violates Constraint 1, which can be identified from the positive value returned from $\textbf{b}_m+\bfm\theta_m-\textbf{A}x(\bfm\theta)$. Therefore, in the next iteration, the model is expanded with the slope of the new critical region, $\nabla \bfm\mu_{\mathcal{B}_2}$, and the first layer parameters are updated as $\textbf{W}^0 = [\nabla \bfm\mu_{\mathcal{B}_1}, \nabla \bfm\mu_{\mathcal{B}_2}] $. To decide the direction of the increment, one can check 
	\begin{equation}
		\nabla \bfm\mu_{\mathcal{B}_2} (-\textbf{B} - \bfm\theta^{d_1}) > \nabla \bfm\mu_{\mathcal{B}_1} (-\textbf{B} -  \bfm\theta^{d_1}).
	\end{equation}	
	As the above statement holds in our example, we expand $\textbf{W}^1=\mathcal{I}$ and $\textbf{W}^2 = v$ with zeros and set $ (\mathcal{I}_{12}, \mathcal{I}_{22}) = (-1,1), v_{2} = 1$.	 
	
	As illustrated in Figure \ref{fig:CL_algorithm2D} - Iteration 1, the updated model can now predict optimal solutions to $\forall \bfm\theta \in CR_1\cup CR_2 $ and is used to predict solutions for all data points sampled in the search direction. This procedure is repeated in Iteration 2 and the resulting model satisfies KKT conditions for all $\bfm \theta$ in the first search direction.
	
	In the next step, another direction is chosen and the process is repeated. When $\bfm\theta$ increases from $\bfm\theta^0$ to $\bfm\theta^{d_3}$, the model predicts optimal solutions to all points except $\bfm\theta\in CR_3$, where the new slope is added and the model now predicts the optimal solution for all points. After this step, the final model becomes,
	\begin{align*}
		f_\mu(\textbf{B}_{\bfm\theta}) &= \textbf{W}^{2T}\sigma(\textbf{W}^{1T}\sigma(\textbf{W}^{0T}(-\textbf{B}-\bfm\theta))) \\
		&=[ \textbf{W}_1^{0T} \textbf{B}_{\bfm\theta} ]^+  + [(\textbf{W}_2^0 - \textbf{W}_1^0)^T \textbf{B}_{\bfm\theta}]^+ +[(\textbf{W}_3^0-\textbf{W}_2^0)\textbf{B}_{\bfm\theta}]^+ -[(\textbf{W}_2^0 - \textbf{W}_4^0) \textbf{B}_{\bfm\theta}]^+,
	\end{align*}
	where $\textbf{B}_{\bfm\theta} = -\textbf{B} - \bfm\theta $.

\renewcommand{\theequation}{C.\arabic{equation}}
\setcounter{equation}{0}

\section{Further Results} 
\subsection{Results with Extreme Demand Characteristics} \label{sec:further_results}

Table \ref{tab:MSE_extreme} reports the average squared KKT violations for all three modeling approaches as well as the worst case violations seen in the dataset by our approach on the dataset prepared for extreme characteristics. The results for CF NN approach does not change overall from the previous case and outperform DNN in all cases. However, DNN models performance reduces significantly, generating 1E+5 to 1E+6 times more violation for all measures overall. For example, KKT1-$\textbf{P}_g$ violations calculated for 6-bus reduces from 2.72E-06 to 7.15E-03 for 6-bus system and from 2.44E-04 to 7.41E+01 for 57-bus system. Finally, as in the previous case,  CF64 model largely outperforms all other models.

\begin{table}
	\centering
	\begin{tabular}{lccccccc}
		\toprule
		&&KKT1-$\textbf{P}_g$ & KKT1-$\bfm\delta$ & KKT2 $(=)$ & KKT2 $(\leq)$ & KKT3 & KKT4\\
		\midrule
		& \ 6bus  &  1.83E-30 & 1.59E-27 & 5.96E-30 & 0.00E+00 & 0.00E+00 & 2.43E-31 \\
		CF64 &  30bus  & 8.81E-31 & 1.04E-27 & 1.36E-30 & 6.83E-32 & 0.00E+00 & 4.83E-32 \\
		&  57bus  & 2.27E-28 & 3.15E-25 & 2.26E-29 & 4.54E-29 & 0.00E+00 & 7.63E-26 \\
		\midrule
		&  \ 6bus  & 1.26E-29 & 1.19E-26 & 5.05E-29 & 1.36E-30 & 0.00E+00 & 2.27E-29 \\
		CF64  &  30bus  & 1.26E-29 & 6.06E-26 & 1.28E-28 & 4.55E-30 & 0.00E+00 & 6.65E-30 \\
		(worst)&  57bus  & 8.53E-27 & 4.14E-23 & 2.79E-27 & 5.98E-26 & 0.00E+00 & 1.17E-22 \\
		\midrule
		&  \ 6bus  & 1.12E-12 & 4.45E-10 & 1.46E-13 & 5.64E-11 & 0.00E+00 & 5.16E-15 \\ 
		CF32 &  30bus  & 1.77E-13 & 1.62E-10 & 1.96E-13 & 3.08E-16 & 0.00E+00 & 1.67E-16 \\ 
		&  57bus  & 1.94E-11 & 4.30E-08 & 3.11E-12 & 8.77E-09 & 0.00E+00 & 1.11E-08 \\ 
		\midrule
		&  \ 6bus  & 5.50E-12 & 3.77E-09 & 2.80E-12 & 1.29E-06 & 0.00E+00 & 6.96E-13 \\ 
		CF32  &  30bus  & 3.64E-12 & 2.35E-08 & 2.19E-11 & 8.88E-14 & 0.00E+00 & 7.06E-14 \\ 
		(worst)&  57bus  & 1.89E-09 & 3.16E-06 & 9.96E-10 & 4.86E-04 & 0.00E+00 & 1.71E-05 \\ 
		\midrule
		& \ 6bus & 7.15E-03 & 8.38E-01 & 3.26E+00 & 1.66E-02 & 5.89E-02 & 1.48E-03 \\
		DNN & 30bus & 1.66E-02 & 1.54E-02 & 3.82E-02 & 4.76E-05 & 8.06E-07 & 2.32E-07 \\
		& 57bus & 7.41E+01 & 8.50E+01 & 3.28E+01 & 9.75E-04 & 4.03E-02 & 1.24E+00 \\
		\midrule
		& \ 6bus & 9.13E-15 & 1.03E-27 & 1.83E-31 & 0.00E+00 & 0.00E+00 & 3.62E-16 \\
		Gurobi 	& 30bus & 4.17E-15 & 3.36E-21 & 6.06E-17 & 0.00E+00 & 0.00E+00 & 9.67E-17 \\
		  &  57bus  & 7.88E-13 & 1.36E-09 & 1.76E-14 & 0.00E+00 & 0.00E+00 & 1.30E-15 \\ 
		\bottomrule
	\end{tabular}
	\caption{Mean Squared KKT Errors on the test set with Extreme Characteristics}
	\label{tab:MSE_extreme}	
\end{table}

\begin{table}[]
     \centering
     \begin{tabular}{cc|ccc|ccc}
     \toprule
 &  &  & Local &  &  & Extreme &  \\
 &  & 25\% & 50\% & 75\% &  25\% & 50\% & 75\% \\
 \midrule
 &  6bus  & 5.46E-09 & 1.02E-08 & 3.09E-08 & -1.42E-14 & 7.32E-10 & 8.74E-09 \\
CF64  &  30bus  & -7.92E-09 & 1.60E-10 & 1.20E-08 & -7.92E-09 & 1.60E-10 & 1.20E-08 \\
 &  57bus  & -1.51E-08 & 1.05E-07 & 7.31E-07 & -5.67E-09 & 1.01E-08 & 1.69E-07 \\
 \midrule
 &  6bus  & 2.97E-07 & 2.13E-06 & 3.92E-06 & 1.38E-06 & 2.62E-06 & 3.98E-06 \\
CF32  &  30bus  & -3.91E-08 & 1.69E-07 & 3.59E-07 & -8.08E-03 & -7.12E-03 & -6.01E-03 \\
 &  57bus  & -4.70E-05 & -1.11E-05 & 2.42E-05 & -2.79E-01 & -2.33E-01 & -1.83E-01 \\
 \bottomrule
     \end{tabular}
\caption{$C(\textbf{P}^g) - \hat C(\textbf{P}^g)$ calculated for the local perturbations and extreme characteristics datasets}
    \label{tab:results_cost}
\end{table}

\subsection{Line Limits}\label{sec:line-limits}
The results in Section \ref{sec:results} and \ref{sec:results_speed} show that our CF model and the proposed learning-via-discovery algorithm can discover the solution function for the problem in (\ref{eq:DC-OPF}). However, the problem in (\ref{eq:DC-OPF}) ignores line limits, which can affect the problem structure 
\begin{subequations}
    \begin{align}
        \min\ &\textbf{P}^{T}\textbf{Q}\textbf{P} + \textbf{C}^T\textbf{P} + \textbf{C}_0,\\
        \text{s.t.: } & \textbf{P}_d +\bfm\theta_e-\textbf{P} -\textbf{B}\bfm\delta = 0,\\
        &\textbf{P}^-\leq \textbf{P} \leq \textbf{P}^+,\\
        -&\textbf{F}^+\leq \mathcal{H}\textbf{B}\bfm\delta \leq \textbf{F}^+,
    \end{align}
    \label{eq:CL_DC-OPF_wF}
\end{subequations}
When the line limits are ignored, the connectivity graph between adjacent critical regions is a simple graph that can be discovered with a simple search algorithm. However, when the line limits are included, we need a better search algorithm. Figure \ref{fig:criticalRegions}a illustrates critical regions similar to the test cases in Section \ref{sec:results}, where each region with a different shade is a distinct critical region. It can be seen that, when the discovery algorithm starts at $\bfm\theta^0$ and $\bfm\theta$ is increased along the $\theta_1$ axis until the last feasible point, the algorithm visits all critical regions. When moving from $\bfm\theta^0$ along the $\theta_2$-axis, the same critical regions are discovered. 

\begin{figure}
    \centering
    \begin{minipage}{.5\textwidth}
        \centering
        \includegraphics[width=.8\linewidth]{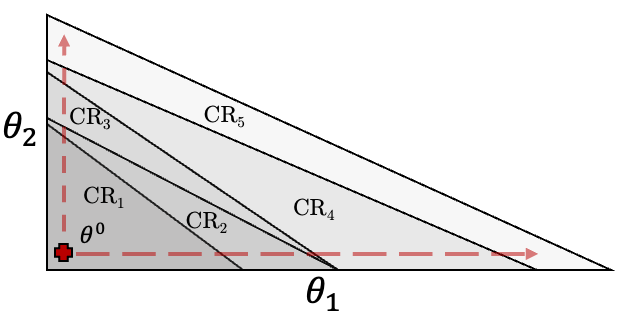}
    \end{minipage}%
    \begin{minipage}{.5\textwidth}
        \centering
        \includegraphics[width=.8\linewidth]{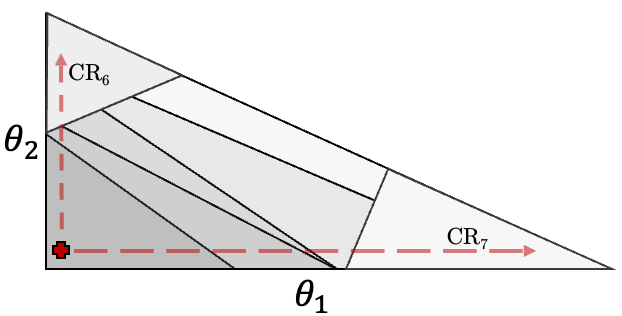}
    \end{minipage}
    \caption{Connectivity of critical regions for problems without (left) and with (right) line limits}
    \label{fig:criticalRegions}
\end{figure}

When line limits are added as constraints, as in \ref{eq:DC-OPF}d, increasing each parameter, $\theta_i$, can trigger a different set of binding constraints, as illustrated in Figure \ref{fig:criticalRegions}b. By increasing $\bfm\theta$ from $\bfm\theta^0$ along the $\theta_1$ axis, the discovery algorithm would discover $CR_1,CR_2, CR_3$, and $CR_6$, respectively, whereas increasing $\bfm\theta$ along $\theta_2$ axis would lead to the discovery of  $CR_1,CR_2$, and $CR_5$ respectively. Notice that in both of these cases, the search algorithm would miss $CR_4$. Moreover, changing the starting point from $\bfm\theta^0$ to another point such as $\bfm\theta'$ and moving along both axes does not guarantee discovering all critical regions.

\begin{figure}
    \centering
    \begin{minipage}{.5\textwidth}
        \centering
        \includegraphics[width=.8\linewidth]{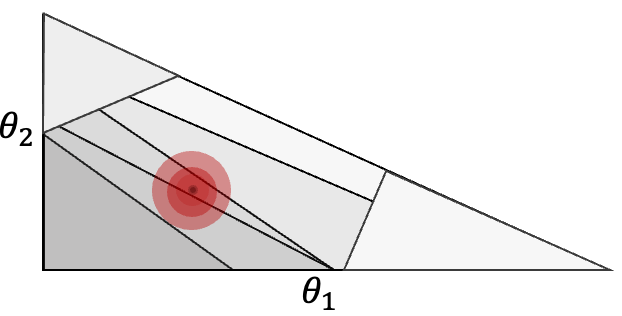}
    \end{minipage}%
    \begin{minipage}{.5\textwidth}
        \centering
        \includegraphics[width=.8\linewidth]{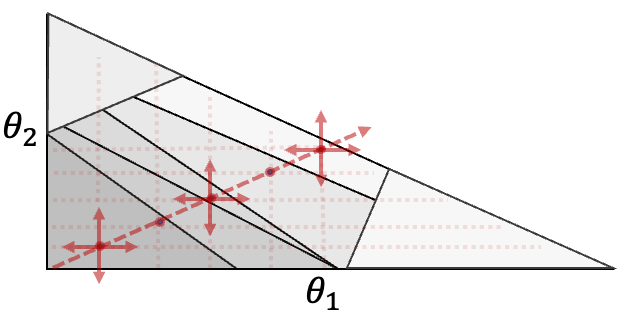}
    \end{minipage}
    \caption{ Local sampling vs. feasible region (left) and updated search algorithm (right) }
    \label{fig:searchAlgo}
\end{figure}

In our preliminary experiments, we sampled a dataset by perturbing a base parameter, $P_{base}^d$ with a multiplier sampled from a distribution, which is the method we used to generate the dataset for Table \ref{tab:MSE_local}. When the discovery algorithm was run by setting $\bfm\theta_e = \textbf{P}_{base}^d$ and moving each $\theta_i$ from 0 to the maximum feasible number, our model could predict the optimal solutions for every test point. However, when $\textbf{P}_{base}$ was changed, our model failed to calculate optimal solutions due to new critical regions that were not discovered by the algorithm.

The plot on the left side of Figure \ref{fig:searchAlgo} illustrates the distribution of input parameters obtained by applying local perturbations and the entire feasible region of the parameters. As shown in the figure, sampling data by applying local perturbations to a base load generates datapoints from three critical regions only. Therefore, this approach does not result in a model that can guarantee solutions outside these three regions. We updated our search algorithm to generate input parameters by moving around a changing  base load, as illustrated on the right side of Figure \ref{fig:searchAlgo}. The updated search algorithm starts with a small value of $k$ and parameters are generated along each axis, starting from the base point until the end of feasible region. The learning-via-discovery algorithm is run to discover critical regions, and to build the model for these regions. Then, $k$ is incrementally increased and the search and discovery steps are repeated.


To test our model, we also updated the test data generation algorithm in a similar fashion to the search algorithm, which is given in Figure \ref{fig:criticalRegions}. The algorithm starts with a base parameter set, $\bfm\theta$, and the set of scaling factors, $[k_1,\dots,k_n]$. With each value of $k$, $\bfm\theta_{base}$ is updated with each value of $k$ and a new parameter set is generated by multiplying the updated base parameter set with a random multiplier, $r_j$, sampled from a Uniform distribution.


\subsection{Test Results with Line Limits}
Our proposed algorithm was run on the 6-bus system to solve the problem in \ref{eq:CL_DC-OPF_wF}. The line limits for each line were set to 200 mW and the base load was set to 140 mW for all load buses. For the test set, we scaled the base load with scaling factors from $k \in [1,1.125,\dots,2]$ as $\bfm\theta^k_{base} = \bfm\theta_{base}\cdot k$ and applied local perturbations same as in earlier sections. We generated 10,000 input parameters for each value of $k$ by multiplying a factor, $r_j \sim Unif(0.6,1.4)$$r_j$ with a modified base load, i.e. $\bfm\theta = r_j \cdot \bfm\theta^k_{base}$. We removed infeasible inputs after solving  each input case using the Gurobi solver. The resulting number of data points are presented in Table \ref{tab:CF_results_wF} under the \# column.

Table \ref{tab:CF_results_wF} presents the results. As seen in the table, for all values of $k\leq 1.625$, the model satisfies KKT optimality conditions with violations less than 5E-10 MSE for all measures. For $k=[1.75,1.875, 2]$, violation errors begin to increase for KKT2 $(\leq)$ and KKT4. An analysis of the cases where our model under-performs indicated that these input parameters were located in undiscovered critical regions close to the boundaries of the feasible region. When these input parameters are removed from the test set, the KKT violations reduced to levels similar to those observed for $k\leq 1.625$.

\begin{table}[]
    \centering
    \begin{tabular}{l|c|c|c|c|c|c|c}
    \toprule
        $k$ & \#  & KKT1 ($\textbf{P}^g$)  & KKT1 ($\boldsymbol\delta$)  & KKT2 $(=)$  & KKT2 ($\leq$) & KKT3 & KKT4 \\
        \midrule
        1.0  & 9924  & 5.40E-13 & 1.63E-10 & 1.98E-13 & 4.91E-15 & 0 & 6.04E-15 \\ 
        1.125 & 10000 & 5.78E-13 & 1.71E-10 & 1.82E-13 & 4.24E-15 & 0 & 3.10E-15 \\ 
        1.25  & 10000 & 6.51E-13 & 1.86E-10 & 1.51E-13 & 2.85E-15 & 0 & 1.27E-15 \\ 
        1.375 & 10000 & 7.12E-13 & 2.01E-10 & 1.23E-13 & 1.70E-15 & 0 & 4.84E-16 \\ 
        1.5 & 9998 & 7.56E-13 & 2.16E-10 & 1.02E-13 & 7.80E-16 & 0 & 2.17E-16 \\ 
        1.625 & 9755 & 8.06E-13 & 2.29E-10 & 9.46E-14 & 3.92E-16 & 0 & 6.59E-16 \\ 
        1.75  & 8797 & 8.16E-13 & 2.33E-10 & 9.37E-14 & 9.77E-11 & 0 & 1.24E-10 \\ 
        1.875 & 7196 & 8.28E-13 & 2.40E-10 & 9.61E-14 & 1.63E-08 & 0 & 5.46E-09 \\ 
        2.0  & 5372 & 8.31E-13 & 2.42E-10 & 1.03E-13 & 2.85E-07 & 0 & 1.06E-07 \\ 
        \midrule
        $1.75^*$ & 8792 & 8.16E-13 & 2.33E-10 & 9.36E-14 & 1.81E-16 & 0 & 1.67E-15 \\ 
        $1.875^*$ & 7170 & 8.28E-13 & 2.40E-10 & 9.53E-14 & 6.81E-14 & 0 & 4.13E-13 \\ 
        $2.0^*$  & 5331 & 8.31E-13 & 2.41E-10 & 1.02E-13 & 5.06E-14 & 0 & 1.35E-12 \\ 
        \bottomrule
        \end{tabular}
    \caption{CF32 Mean Squared KKT Errors on the 6-bus test set with line limits}
    \label{tab:CF_results_wF}
\end{table}


\end{document}